\documentclass[a4paper,10pt]{amsart}
\numberwithin{equation}{section}
\allowdisplaybreaks[1]

\usepackage{mathtools,amssymb,bm,color,setspace,enumitem,eucal,mathrsfs}
\usepackage[noadjust]{cite}
\usepackage[inner=20mm, outer=20mm, top=23mm, bottom=23mm, head=10mm, foot=10mm]{geometry}

\usepackage{array}
\newcolumntype{C}{>{$}c<{$}} %Defines math mode in tabular

\usepackage[largesc,theoremfont]{newtxtext}      % mathptmx is now considered obsolete...
\usepackage[libertine,cmbraces,varbb]{newtxmath} % but the replacement greek math looks crap (hence libertine)...
\begingroup
  \makeatletter
  \@for\theoremstyle:=definition,remark,plain\do{%
    \expandafter\g@addto@macro\csname th@\theoremstyle\endcsname{%
      \addtolength\thm@preskip{.5\baselineskip plus .2\baselineskip minus .2\baselineskip}
      \addtolength\thm@postskip{.5\baselineskip plus .2\baselineskip minus .2\baselineskip}
    }%
  }
\endgroup
\usepackage{tikz}
\tikzstyle{wt}=[circle, fill=black, inner sep=1.2pt, outer sep=0pt, minimum size=2pt]  % a weight

\usepackage[colorlinks=true,citecolor=red,linkcolor=blue,urlcolor=blue]{hyperref} % load _before_ cleveref!
\usepackage[capitalise,noabbrev]{cleveref} % great package: use \cref and \Cref for (almost) all \refs except \eqref

\setitemize{leftmargin=*}   % Removes margin from itemized lists (enumitem package)
\setenumerate{leftmargin=*, % Removes margin for enumerate lists (enumitem package)
 label=\textup{(\arabic*)}} % Makes labels upright lowercase arabic numerals

\theoremstyle{definition}
\newtheorem*{definition}{Definition}%[section]
\theoremstyle{plain}
\newtheorem{mtheorem}{Main Theorem}
\newtheorem{theorem}{Theorem}[section]
\newtheorem{lemma}[theorem]{Lemma}%[section]
\newtheorem{proposition}[theorem]{Proposition}%[section]
\newtheorem*{remark}{Remark}%[section]
\newtheorem*{example}{Example}
\renewcommand{\ge}{\geqslant} % never use \geq or \geqslant, use \ge which is globally defined to be one or the other
\renewcommand{\le}{\leqslant} % same for \leq and \leqslant

\newcommand{\cc}{\mathsf{c}}   % central charge
\newcommand{\ii}{\mathfrak{i}} % imaginary unit
\newcommand{\kk}{\mathsf{k}}   % level
\newcommand{\qq}{\mathsf{q}}   % modular parameter
\newcommand{\zz}{\mathsf{z}}   % modular parameter

\DeclareMathOperator{\res}{Res}

\DeclarePairedDelimiter{\brac}{\lparen}{\rparen}   % use \brac for (...) and \brac* to automatically scale the ( and )
\DeclarePairedDelimiter{\sqbrac}{\lbrack}{\rbrack} % use \sqbrac[\big] for \bigl(...\bigr) etc...
\DeclarePairedDelimiter{\set}{\lbrace}{\rbrace}
\newcommand{\st}{\mspace{5mu} {:} \mspace{5mu}}    % "such that" in sets
\DeclarePairedDelimiter{\abs}{\lvert}{\rvert}
\DeclarePairedDelimiter{\norm}{\lVert}{\rVert}
\DeclarePairedDelimiter{\ang}{\langle}{\rangle}    % spans or ideals or ...

\DeclarePairedDelimiterX{\comm}[2]{\lbrack}{\rbrack}{#1 , #2}  % commutators
\DeclarePairedDelimiterX{\acomm}[2]{\lbrace}{\rbrace}{#1 , #2} % anticommutators
\DeclarePairedDelimiterX{\inner}[2]{\langle}{\rangle}{#1 , #2} % scalar products
\DeclarePairedDelimiterX{\super}[2]{\lparen}{\rparen}{#1 \delimsize\vert \mathopen{} #2} % for super args (m|n)

\newcommand{\Ra}{\Rightarrow}
\newcommand{\lra}{\longrightarrow}
\newcommand{\ira}{\hookrightarrow}    % for injections
\newcommand{\lira}{\ensuremath{\lhook\joinrel\relbar\joinrel\rightarrow}} % long injection
\newcommand{\dses}[3]{0 \lra #1 \lra #2 \lra #3 \lra 0} % displayed ses

\newcommand{\categ}[1]{\mathscr{#1}} % categories

\newcommand{\fld}[1]{\mathbb{#1}}    % font for fields
\newcommand{\ZZ}{\fld{Z}}
\newcommand{\RR}{\fld{R}}
\newcommand{\CC}{\fld{C}}

\newcommand{\grp}[1]{\mathsf{#1}}              % font for groups

\newcommand{\SLG}[2]{\grp{#1}_{#2}}            % Lie groups like SU(2)
\newcommand{\alg}[1]{\mathfrak{#1}}                      % font for Lie algebras
\newcommand{\finite}[1]{#1}                              % could use \overline to indicate finite-dim algebras
\newcommand{\affine}[1]{\widehat{#1}}

\newcommand{\SLA}[2]{\finite{\alg{#1}}{}_{#2}}           % Lie algebras like sl(2) (the {} prevents a superscript being too high)
\newcommand{\SLSA}[3]{\finite{\alg{#1}} \super{#2}{#3}}  % Lie superalgebras like gl(1|1)
\newcommand{\AKMA}[2]{\affine{\alg{#1}}{}_{#2}}          % Kac-Moody algebras
\newcommand{\AKMSA}[3]{\affine{\alg{#1}} \super{#2}{#3}} % Kac-Moody superalgebras

\newcommand{\envalg}[1]{\mathsf{U}\brac{#1}}                % Universal enveloping (super)algebras
\newcommand{\envalgk}[1]{\mathsf{U}_{\kk} \brac{#1}}        % U([g,g]) / <K-k.1>
\newcommand{\ideal}[1]{\ang*{#1}}
\newcommand{\wun}{\vvmathbb{1}}                             % unit of UEAs, needs newtxmath

\newcommand{\g}{\finite{\alg{g}}} %\newcommand{\g}{\mathfrak{g}}  % generic Lie algebra
\newcommand{\h}{\finite{\alg{h}}}                                 % Cartan subalgebras?
\newcommand{\ag}{\affine{\g}}                                     % generic affine Lie algebra
\newcommand{\ah}{\affine{\h}}                                     % affine Cartan subalgebras

\newcommand{\sltwo}{\SLA{sl}{2}}
\newcommand{\asltwo}{\AKMA{sl}{2}}
\newcommand{\osp}{\SLSA{osp}{1}{2}}
\newcommand{\aosp}{\AKMSA{osp}{1}{2}}
\newcommand{\raosp}{{}^{\R} \AKMSA{osp}{1}{2}}

\DeclareMathOperator{\ad}{ad}                   % adjoint action (algebra on algebra)

\newcommand{\cas}{Q}                            % quadratic Casimir
\newcommand{\scas}{\Sigma}                      % super-Casimir of osp(1|2)
\newcommand{\wlat}{\grp{P}}                     % weight lattice
\newcommand{\rlat}{\grp{Q}}                     % root lattice
\newcommand{\sroot}{\alpha}                     % simple root (rank 1)
\newcommand{\fwt}{\omega}                       % fundamental weight (rank 1)
\newcommand{\wvec}{\rho}                        % Weyl vector
\newcommand{\killing}[2]{\kappa \brac{#1 , #2}} % the Killing form or the next best thing
\newcommand{\bilin}[2]{\brac{#1 , #2}}          % the induced bilinear form on the weight space
\newcommand{\cox}{\mathsf{h}}                   % Coxeter number
\newcommand{\dcox}{\cox^{\vee}}                 % dual Coxeter number

\newcommand{\Mod}[1]{\mathcal{#1}}   % font for modules
\newcommand{\aMod}[1]{\affine{\mathcal{#1}}}

\newcommand{\conj}{\mathsf{w}}       % conjugation functor

\newcommand{\rel}[1]{\affine{\RelVerSymb}_{#1}}       % relaxed Verma for general AKMAs

\newcommand{\nsimpset}[1]{\Lambda(#1)}                % set of central characters for which sl2 dense modules are reducible
\newcommand{\nsimpseto}[1]{\Lambda'(#1)}              % set of central characters for which osp dense modules are reducible

\newcommand{\VerSymb}{\Mod{V}}    % symbol for Verma mods
\newcommand{\IrrSymb}{\Mod{L}}    % for simple hw mods
\newcommand{\FinSymb}{\Mod{V}}    % for finite-dim hw mods
\newcommand{\RelVerSymb}{\Mod{R}} % for relaxed Verma mods
\newcommand{\RelSymb}{\Mod{E}}    % for (generically) simple relaxed mods
\newcommand{\MaxSymb}{\Mod{J}}    % for maximal proper submodules
\newcommand{\NoIntSymb}{\Mod{I}}  % for maximal proper submodules whose intersection with the top space is zero

\newcommand{\fver}[1]{\finite{\VerSymb}_{#1}}         % Verma module over a finite-dim algebra
\newcommand{\ffin}[1]{\finite{\FinSymb}_{#1}}         % finite-dim irreducible hw module
\newcommand{\fden}[2]{\finite{\RelVerSymb}_{#1; #2}}  % (generically irreducible) dense module
\newcommand{\fcoh}[1]{\finite{\RelVerSymb}_{#1}}      % coherent family

\newcommand{\aver}[1]{\affine{\VerSymb}_{#1}}         % Verma module over an affine algebra
\newcommand{\afin}[1]{\affine{\FinSymb}_{#1}}         % parabolic Verma module
\newcommand{\airr}[1]{\affine{\IrrSymb}_{#1}}         % irreducible hw module
\newcommand{\arver}[2]{\affine{\RelVerSymb}_{#1; #2}} % relaxed Verma module
\newcommand{\arel}[2]{\affine{\RelSymb}_{#1; #2}}     % generically irreducible relaxed module
\newcommand{\armax}[2]{\affine{\MaxSymb}_{#1; #2}}    % maximal proper submodule of affine relaxed module
\newcommand{\arint}[2]{\affine{\NoIntSymb}_{#1; #2}}  % max submod of aff rel mod whose top space intersection is zero
\newcommand{\avcoh}[1]{\affine{\RelVerSymb}_{#1}}     % relaxed Verma coherent family
\newcommand{\acoh}[1]{\affine{\RelSymb}_{#1}}         % generically simple affine coherent family
\newcommand{\parity}[1]{\overline{#1}}                 % parity of a vector
\newcommand{\parrev}{\Pi}                              % parity-reversal functor

\newcommand{\fvero}[1]{\finite{\VerSymb}_{#1}}         % Verma module over a finite-dim algebra
\newcommand{\ffino}[1]{\finite{\FinSymb}_{#1}}         % finite-dim irreducible hw module
\newcommand{\fdeno}[2]{\finite{\RelVerSymb}_{#1; #2}}  % (generically irreducible) dense module
\newcommand{\NS}{\textup{\textsc{ns}}}
\newcommand{\averns}[1]{{}^{\NS} \affine{\VerSymb}_{#1}}         % Verma module over an affine algebra
\newcommand{\afinns}[1]{{}^{\NS} \affine{\FinSymb}_{#1}}         % parabolic Verma module
\newcommand{\airrns}[1]{{}^{\NS} \affine{\IrrSymb}_{#1}}         % irreducible hw module
\newcommand{\arverns}[2]{{}^{\NS} \affine{\RelVerSymb}_{#1; #2}} % relaxed Verma module
\newcommand{\arelns}[2]{{}^{\NS} \affine{\RelSymb}_{#1; #2}}     % generically irreducible relaxed module
\newcommand{\armaxns}[2]{{}^{\NS} \affine{\MaxSymb}_{#1; #2}}    % maximal proper submodule of affine relaxed module
\newcommand{\arintns}[2]{{}^{\NS} \affine{\NoIntSymb}_{#1; #2}}  % max submod of aff rel mod whose top space intersection is zero
\newcommand{\avcohns}[1]{{}^{\NS} \affine{\RelVerSymb}_{#1}}     % relaxed Verma coherent family
\newcommand{\acohns}[1]{{}^{\NS} \affine{\RelSymb}_{#1}}         % generically simple affine coherent family
\newcommand{\R}{\textup{\textsc{r}}}
\newcommand{\averr}[1]{{}^{\R} \affine{\VerSymb}_{#1}}         % Verma module over an affine algebra
\newcommand{\afinr}[1]{{}^{\R} \affine{\FinSymb}_{#1}}         % parabolic Verma module
\newcommand{\airrr}[1]{{}^{\R} \affine{\IrrSymb}_{#1}}         % irreducible hw module
\newcommand{\arverr}[2]{{}^{\R} \affine{\RelVerSymb}_{#1; #2}} % relaxed Verma module
\newcommand{\arelr}[2]{{}^{\R} \affine{\RelSymb}_{#1; #2}}     % generically irreducible relaxed module
\newcommand{\armaxr}[2]{{}^{\R} \affine{\MaxSymb}_{#1; #2}}    % maximal proper submodule of affine relaxed module
\newcommand{\arintr}[2]{{}^{\R} \affine{\NoIntSymb}_{#1; #2}}  % max submod of aff rel mod whose top space intersection is zero
\newcommand{\avcohr}[1]{{}^{\R} \affine{\RelVerSymb}_{#1}}     % relaxed Verma coherent family
\newcommand{\acohr}[1]{{}^{\R} \affine{\RelSymb}_{#1}}         % generically simple affine coherent family
\DeclareMathOperator{\tr}{tr}

\newcommand{\Gr}[1]{\sqbrac[\big]{#1}}                 % element of a Grothendieck group/ring
\newcommand{\tGr}[1]{\sqbrac{#1}}                      % element of a Grothendieck group/ring

\newcommand{\traceover}[1]{\tr_{\raisebox{-2pt}{$\scriptstyle #1$}}} % for ch[M] = tr_M ...

\DeclareMathOperator{\chmap}{ch}
\DeclareMathOperator{\sfnmap}{s}

\newcommand{\ch}[1]{\chmap \Gr{#1}}                    % character
\newcommand{\fch}[2]{\ch{#1} \brac{#2}}                % character as function of q and ...
\newcommand{\sfn}[2]{\sfnmap_{#1} \Gr{#2}}             % string function
\newcommand{\fsfn}[3]{\sfn{#1}{#2} \brac{#3}}
\newcommand{\tsfn}[2]{\sfnmap_{#1} \tGr{#2}}
\newcommand{\tfsfn}[3]{\tsfn{#1}{#2} \brac{#3}}

\newcommand{\rsfn}[3]{\sfnmap_{#1}^{#2} \Gr{#3}}             % limiting R-string function
\newcommand{\frsfn}[4]{\rsfn{#1}{#2}{#3} \brac{#4}}
\newcommand{\trsfn}[3]{\sfnmap_{#1}^{#2} \tGr{#3}}

\newcommand{\evsfn}[2]{\rsfn{#1}{+}{#2}}                     % even R-string function
\newcommand{\fevsfn}[3]{\evsfn{#1}{#2} \brac{#3}}

\newcommand{\odsfn}[2]{\rsfn{#1}{-}{#2}}                     % odd R-string function
\newcommand{\fodsfn}[3]{\odsfn{#1}{#2} \brac{#3}}

\newcommand{\jth}[1]{\vartheta_{#1}}                   % Jacobi theta
\newcommand{\fjth}[2]{\jth{#1} \brac{#2}}              % Jacobi theta as a function of z, q (or zeta, tau)

\newcommand{\VOA}[1]{\mathsf{#1}} % font for VOAs

\newcommand{\slnvoa}[2]{\VOA{L}_{#1}(\SLA{sl}{#2})}         % simple affine VOA of A-type (#1 is k, #2 is n)
\newcommand{\slvoa}[1]{\slnvoa{#1}{2}}
\newcommand{\ospmnvoa}[3]{\VOA{L}_{#1}(\SLSA{osp}{#2}{#3})} % simple affine SVOA of BCD-types
\newcommand{\ospvoa}[1]{\ospmnvoa{#1}{1}{2}}

\newcommand{\cft}{conformal field theory}
\newcommand{\cfts}{conformal field theories}

\newcommand{\lcfts}{logarithmic conformal field theories}
\newcommand{\voa}{vertex operator algebra}
\newcommand{\voas}{vertex operator algebras}

\newcommand{\svoa}{vertex operator superalgebra}
\newcommand{\svoas}{vertex operator superalgebras}

\newcommand{\uea}{universal enveloping algebra}

\newcommand{\pbw}{Poincar\'{e}{-}Birkhoff{-}Witt}
\newcommand{\bgg}{Bern\v{s}te\u{\i}n{-}Gel'fand{-}Gel'fand}
\newcommand{\wzw}{Wess-Zumino-Witten}

\newcommand{\lw}{lowest-weight}
\newcommand{\lwv}{\lw{} vector}
\newcommand{\lwvs}{\lw{} vectors}

\newcommand{\hw}{highest-weight}
\newcommand{\hwv}{\hw{} vector}
\newcommand{\hwvs}{\hw{} vectors}
\newcommand{\hwm}{\hw{} module}
\newcommand{\hwms}{\hw{} modules}

\newcommand{\lhss}{left-hand sides}
\newcommand{\rhs}{right-hand side}

\newcommand{\ns}{Neveu-Schwarz}

\title{Relaxed highest-weight modules I: rank $1$ cases}

\author[K~Kawasetsu]{Kazuya Kawasetsu}
\address[Kazuya Kawasetsu]{
School of Mathematics and Statistics \\
University of Melbourne \\
Parkville, Australia, 3010.
}
\email{kazuya.kawasetsu@unimelb.edu.au}

\author[D~Ridout]{David Ridout}
\address[David Ridout]{
School of Mathematics and Statistics \\
University of Melbourne \\
Parkville, Australia, 3010.
}
\email{david.ridout@unimelb.edu.au}

\begin{document}

\begin{abstract}
	Relaxed highest-weight modules play a central role in the study of many important vertex operator (super)algebras and their associated (logarithmic) conformal field theories, including the admissible-level affine models.  Indeed, their structure and their (super)characters together form the crucial input data for the standard module formalism that describes the modular transformations and Grothendieck fusion rules of such theories.  In this article, character formulae are proved for relaxed \hwms{} over the simple admissible-level affine vertex operator superalgebras associated to $\sltwo$ and $\osp$.  Moreover, the structures of these modules are specified completely.  This proves several conjectural statements in the literature for $\sltwo$, at arbitrary admissible levels, and for $\osp$ at level $-\frac{5}{4}$.  For other admissible levels, the $\osp$ results are believed to be new.
\end{abstract}

\maketitle

\onehalfspacing

\section{Introduction} \label{sec:intro}

Relaxed \hwms{} are a generalisation of the usual \hwms{} that are playing an increasingly important role in the representation theory of \svoas{} and their associated \cfts{}.  The name comes from the work of Feigin, Semikhatov and Tipunin \cite{FeiEqu98,FeiRes98} on the implications of the well known coset construction of the $N=2$ superconformal algebras for the representation theory of the affine Kac-Moody algebra $\asltwo$ (see \cite{SatEqu16,SatMod17,CreN=218} for recent progress on this).  In this work, they \emph{relax} the definition of a \hwv{} so that it need not be annihilated by the positive root vector of the horizontal subalgebra.  The notion of a relaxed \hwm{} has since been generalised \cite{RidRel15} to infinite-dimensional Lie superalgebras admitting a conformal grading.

A relaxed \hwm{} may therefore be described as a generalised \hwm{} obtained by inducing a weight module over the horizontal subalgebra.  The notion is similar to, but more general than, a parabolic \hwm{} because the space of ground states (equivalently, the module that one induces from) is not required to be finite-dimensional nor simple.  It seems that such modules were first considered in the vertex algebra literature in \cite{AdaVer95}, where the simple ones were classified for the admissible-level affine \voas{} $\slvoa{\kk}$.  They have also appeared in the physics literature as integral components of the $\SLG{SL}{2}(\RR)$ \wzw{} model \cite{MalStr01} and through requiring closure under fusion and cosets in $\slvoa{\kk}$ \cfts{} \cite{GabFus01,RidSL210,RidFus10,CreCos13}.  More recently, relaxed \hwms{} over $\slnvoa{\kk}{3}$ at admissible levels have also begun to receive attention \cite{AdaRea16,AraWei16}.

There are two observations relating to relaxed \hwms{} which we find compelling as arguments for their continued study.  First, they provide the most natural setting in which to study weight modules over \voas{} using Zhu algebra technology \cite{ZhuMod96}.  Second, they are an essential ingredient in many applications of the standard module formalism \cite{CreLog13,RidVer14} to the modular properties of \lcfts{}.  This formalism, which originated in \cite{CreRel11,CreWAl11}, identifies a set of standard modules, which need not be simple, from which all simple modules may be constructed using resolutions and all Grothendieck fusion rules may be computed using a variant of the celebrated Verlinde formulae of rational \cft{} \cite{VerFus88,HuaVer08}.  These standard modules turn out to be relaxed \hwms{} for admissible-level $\slvoa{\kk}$ \cite{CreMod12,CreMod13}, admissible-level $\ospvoa{\kk}$ \cite{RidAdm17,CreCos18}, and the bosonic ghost system \cite{RidBos14}.  We expect that this observation will generalise appropriately to higher-rank affine \voas{}.

One of the main inputs of the standard module formalism is a character formula for the standard modules.  For admissible-level $\slvoa{\kk}$ and $\ospvoa{\kk}$, this means determining the characters of the relaxed \hwms{}.  The characters of the reducible relaxed $\slvoa{\kk}$-modules were first computed in \cite{CreMod12,CreMod13}, but the corresponding simple $\slvoa{\kk}$-characters were only noted to follow from some unproven assertions in \cite{SemEmb97,FeiEqu98}.  Similarly, the simple $\ospvoa{-5/4}$-characters were only conjectured in \cite{RidAdm17}.

This unsatisfactory state of affairs has recently been partially rectified by Adamovi\'{c} in \cite{AdaRea17}.  There, he explicitly constructs the relaxed \hw{} $\slvoa{\kk}$- and \ns{} $\ospvoa{-5/4}$-modules using a clever free field realisation that effectively inverts the quantum hamiltonian reduction, see also \cite{SemInv94}.  While this construction leads to straightforward determinations of the characters, it is not obvious that the resulting modules are generically simple.  The simple characters therefore only follow when there are no coincidences of conformal weights, modulo $1$.  Note that a similar character formula had been previously proven for certain \emph{critical-level} relaxed \hw{} $\asltwo$-modules in \cite{AdaLie07}.

%In the arXiv version of \cite{SatMod17}, Sato also explicitly constructs relaxed \hw{} $\slvoa{\kk}$-modules using a twisted localisation functor that generalises that of Enright \cite{EnrFun79,DeoCon80} (it may be regarded as an affine analogue of that of Mathieu \cite{MatCla00}).  However, it is again not obvious that the result is generically simple and so simplicity is established by appealing to the conjectural character formulae of \cite{CreMod13}.  These formulae were then used to deduce the characters and modular transformations of the ``typical'' simple \hwms{} of the (non-unitary) $N=2$ superconformal minimal models \svoas{}, see also \cite{CreN=218}.  In the published version, this construction is missing and the $N=2$ characters are instead deduced from the recently determined characters of certain \hw{} $\AKMSA{sl}{2}{1}$-modules \cite{GorCha15}.  One should now be able to deduce relaxed $\asltwo$-characters from these $N=2$ characters using the equivalences established in \cite{FeiEqu98,SatEqu16}, but this does not seem to have been noted in \cite{SatMod17}.

A second main input to the standard module formalism, and more widely to constructing projective covers for the \hw{} simples, is the determination of the structure of the non-simple standard modules.  This structure is needed to construct the resolutions that relate the non-standard simples to standards and thereby enable the study of the modularity of the simple modules of the theory.  Again, these structures were stated without proof and used extensively in \cite{CreMod12,CreMod13,RidAdm17}.

Our aim in this work is to rigorously prove the character formulae and structural results of \cite{CreMod12,CreMod13,RidAdm17} for all admissible levels.  Instead of an explicit construction, we develop the structure theory of ``relaxed Verma modules'' and their simple quotients over both $\asltwo$ and $\aosp$, the latter in both its \ns{} and Ramond incarnations.  The first main result (see below) is a means to compute the character of an arbitrary simple relaxed \hwm{} from that of an associated simple (usual) \hwm{}.  When the latter character is known, for example through the Kac-Wakimoto formula for admissible-level \hwms{} \cite{KacMod88,KacMod88b}, we can thereby deduce the required relaxed characters.  This is our second main result.  The third settles the structures of the non-simple relaxed modules in terms of non-split short exact sequences.  The key technical tools we use to prove these results are a generalisation of Mathieu's coherent families \cite{MatCla00} to a relaxed affine setting and a study of a Shapovalov-like form on the resulting relaxed coherent families.  %We are currently generalising our results to admissible-level $\slnvoa{\kk}{3}$, where one has the recent classification results of \cite{AraWei16}, with the aim of proving much more general results in the future \cite{KawCla18,KawSL318}.

\subsection{Main results} \label{sec:resultsummary}

We divide our conclusions into three main results.  The first applies to general simple relaxed \hw{} $\asltwo$- and $\aosp$-modules of fixed level $\kk$.  These $\asltwo$-modules are denoted by $\arel{\lambda}{q}$, where $\lambda$ is a coset in the quotient of the weight space of $\sltwo$ by its root lattice and $q$ is the eigenvalue of the quadratic Casimir of $\sltwo$ on the ground states (see \cref{sec:rhwsl2}).  The $\aosp$-modules fall into \ns{} and Ramond sectors and are denoted by $\arelns{\lambda}{\sigma}$ and $\arelr{\lambda}{q}$, respectively.  Here, $\lambda$ is a coset in the quotient of the weight space of $\osp$ by its even root lattice and $\sigma$ is the eigenvalue of the super-Casimir of $\osp$ on the even ground states (see \cref{sec:rhwosp}).  In the Ramond sector, $q$ continues to refer to the $\sltwo$-Casimir eigenvalue, now understood with respect to the usual embedding $\sltwo \ira \osp$.

We say that a weight $\asltwo$- or $\aosp$-module $\aMod{M}$ is \emph{stringy} if its (non-zero) string functions $\tsfn{\nu}{\aMod{M}}$ are independent of the $\sltwo$- or $\osp$-weight $\nu$, respectively.  An $\aosp$-module is \emph{\R-stringy} if its string functions only depend on whether the corresponding $\osp$-weight is even or odd.  Finally, let $\airr{\mu}^+$, $\airrns{\mu}^+$ and $\airrr{\mu}^+$ denote the level-$\kk$ simple \hw{} $\asltwo$-, \ns{} $\aosp$- and Ramond $\aosp$-module whose \hwv{} is even with $\sltwo$- and $\osp$-weight $\mu$, respectively.  We can now state our first main result, combining \cref{thm:simplestringy,thm:simpstrfns,thm:nsimpstrfns} for $\asltwo$ with \cref{prop:simplestringyosp,thm:simpstrfnsosp,thm:nsimpstrfnsosp} for $\aosp$.
\begin{mtheorem}
	\leavevmode
	\begin{itemize}
		\item The relaxed \hw{} $\asltwo$-module $\arel{\lambda}{q}$ is stringy and its string functions are given by
		\begin{equation}
			\sfn{\nu}{\arel{\lambda}{q}} = \lim_{m \to \infty} \sfn{-\mu-m\sroot}{\airr{-\mu-\sroot}^+}, \qquad \text{for all}\ \nu \in \lambda,
		\end{equation}
		where $\sroot$ is the simple root of $\sltwo$ and $\mu$ denotes any solution of $\bilin{\mu}{\mu+\sroot} = q$, if $\sqrt{1+2q} \notin \ZZ$, and the maximal such solution (with respect to the real part of its Dynkin label), if $\sqrt{1+2q} \in \ZZ$.
		\item The \ns{} relaxed \hw{} $\aosp$-module $\arelns{\lambda}{\sigma}$ is stringy.  For $\sigma \notin \ZZ + \frac{1}{2}$, its string functions are given by
		\begin{equation}
			\sfn{\nu}{\arelns{\lambda}{\sigma}} = \lim_{m \to \infty} \sfn{-\mu-m\fwt}{\airrns{-\mu-\fwt}^+}, \qquad \text{for all}\ \nu \in \lambda \cup (\lambda+\fwt),
		\end{equation}
		where $\fwt$ is the (odd) simple root of $\osp$ and $\mu = \brac{\sigma - \frac{1}{2}} \fwt$.  This identity also holds for $\sigma \in \ZZ + \frac{1}{2}$ when $\sigma>0$.  However, when $\sigma<0$, we must replace the string function on the \rhs{} by $\tsfn{\mu-m\fwt}{\airrns{\mu}^+}$.
		\item The Ramond relaxed \hw{} $\aosp$-module $\arelr{\lambda}{q}$ is \R-stringy and its string functions are given by
		\begin{equation}
			\fsfn{\nu}{\arelr{\lambda}{q}}{\qq} =
			\begin{cases*}
				\lim_{m \to \infty} \sfn{-\mu-2m\fwt}{\airrr{-\mu-2\fwt}^+}, & for all $\nu \in \lambda$, \\
				\lim_{m \to \infty} \sfn{-\mu-(2m+1)\fwt}{\airrr{-\mu-2\fwt}^+}, & for all $\nu \in \lambda+\fwt$,
			\end{cases*}
		\end{equation}
		where $\mu$ now denotes any solution of $\bilin{\mu}{\mu+2\fwt} = q$, if $\sqrt{1+2q} \notin \ZZ$, and the maximal such solution, if $\sqrt{1+2q} \in \ZZ$.
	\end{itemize}
\end{mtheorem}

Our second main result concerns the specialisation of the first to modules over the simple admissible-level \svoas{} $\slvoa{\kk}$ and $\ospvoa{\kk}$.  For $\sltwo$, the level $\kk$ is said to be admissible if $\kk+2 = \frac{u}{v}$, where $u \in \ZZ_{\ge 2}$, $v \in \ZZ_{\ge 1}$ and $\gcd \set{u,v} = 1$.  Only the $\arel{\lambda}{q}$ with
\begin{equation}
	q = q_{r,s} = \frac{(vr-us)^2 - v^2}{2v^2}, \qquad r=1,\dots,u-1 \quad \text{and} \quad s=1,\dots,v-1
\end{equation}
define $\slvoa{\kk}$-modules \cite{AdaVer95,RidRel15}.  For $\osp$, $\kk$ is admissible if $\kk+\frac{3}{2} = \frac{u}{2v}$, where $u \in \ZZ_{\ge 2}$, $v \in \ZZ_{\ge 1}$, $\frac{1}{2}(u-v) \in \ZZ$ and $\gcd \set{\frac{1}{2}(u-v),v} = 1$.  Moreover, the $\arelns{\lambda}{\sigma}$ and $\arelr{\lambda}{q}$ are only $\ospvoa{\kk}$-modules if
\begin{equation}
	\begin{aligned}
		\sigma = \sigma_{r,s} &= \frac{vr-us}{2v}, &&& r&=1,\dots,u-1, & s&=1,\dots,v-1, & &\text{and} & r-s &\in 2\ZZ+1, \\
		\text{and} \qquad q = q_{r,s} &= \frac{(vr-us)^2 - 4v^2}{8v^2}, &&& r&=1,\dots,u-1, & s&=1,\dots,v-1, & &\text{and} & r-s &\in 2\ZZ,
	\end{aligned}
\end{equation}
respectively.  In both cases, $\slvoa{\kk}$ and $\ospvoa{\kk}$, the set of these relaxed modules is empty if $v=1$ ($\kk \in \ZZ_{\ge 0}$).

\cref{thm:chsl2,thm:NSchosp,thm:Rchosp} now give the characters of these relaxed $\slvoa{\kk}$- and $\ospvoa{\kk}$-modules, proving the conjectural formulae of \cite{CreMod12,CreMod13,RidAdm17}.  As far as we know, the formulae for $\ospvoa{\kk}$ with $\kk$ admissible and not equal to $-\frac{5}{4}$ are new.
\begin{mtheorem}
	We have the following character formulae:
	\begin{subequations}
		\begin{align}
			\fch{\arel{\lambda}{q_{r,s}}}{\zz;\qq} &= \zz^{\lambda} \frac{\chi^{\textup{Vir}}_{r,s}(\qq)}{\eta(\qq)^2} \sum_{n \in \ZZ} (\zz^{\sroot})^n, \\
			\fch{\arelns{\lambda}{\sigma_{r,s}}}{\zz;\qq} &= \zz^{\lambda} \frac{\chi^{N=1}_{r,s}(\qq)}{\eta(\qq)^2} \sqrt{\frac{\fjth{2}{1;\qq}}{2 \eta(\qq)}} \sum_{n \in \ZZ} (\zz^{\fwt})^n, \\
			\fch{\arelr{\lambda}{q_{r,s}}}{\zz;\qq} &= \zz^{\lambda} \sqbrac*{\frac{\chi^{N=1}_{r,s}(\qq)}{2 \eta(\qq)^2} \sqrt{\frac{\fjth{3}{1;\qq}}{\eta(\qq)}} \sum_{n \in \ZZ} (\zz^{\fwt})^n + \frac{\overline{\chi}^{N=1}_{r,s}(\qq)}{2 \eta(\qq)^2} \sqrt{\frac{\fjth{4}{1;\qq}}{\eta(\qq)}} \sum_{n \in \ZZ} (-\zz^{\fwt})^n}.
		\end{align}
	\end{subequations}
	Here, $\chi^{\textup{Vir}}_{r,s}$, $\chi^{N=1}_{r,s}$ and $\overline{\chi}^{N=1}_{r,s}$ denote the Virasoro minimal model character \eqref{eq:virch}, the $N=1$ superconformal minimal model character \eqref{eq:RchN=1} or \eqref{eq:NSchN=1}, and the $N=1$ superconformal minimal model supercharacter \eqref{eq:NSschN=1}, respectively.
\end{mtheorem}

The final main result concerns the structure of the non-simple relaxed $\slvoa{\kk}$- and $\ospvoa{\kk}$-modules.  Up to isomorphism, these are the $\arel{\lambda}{q_{r,s}}$, $\arelns{\lambda}{\sigma_{r,s}}$ and $\arelr{\lambda}{q_{r,s}}$ whose coset $\lambda$ contains $\mu_{r,s}$, where
\begin{equation}
	\mu_{r,s} = \frac{1}{2} \brac*{r-1 - \frac{u}{v} s} \sroot, \qquad
	\mu_{r,s} = \frac{1}{2} \brac*{r-1 - \frac{u}{v} s} \fwt \qquad \text{and} \qquad
	\mu_{r,s} = \frac{1}{2} \brac*{r-2 - \frac{u}{v} s} \fwt,
\end{equation}
respectively.  With this, the structures are characterised by \cref{thm:esvoa,thm:esvoaosp}.
\begin{mtheorem}
	We have the following non-split short exact sequences:
	\begin{subequations}
		\begin{gather}
			\dses{\airr{\mu_{r,s}}^+}{\arel{\mu_{r,s}}{q_{r,s}}}{\conj \airr{\mu_{u-r,v-s}}^+}, \\
			\dses{\airrns{\mu_{r,s}}^+}{\arelns{\mu_{r,s}}{\sigma_{r,s}}}{\parrev \conj \airrns{\mu_{u-r,v-s}}^+}, \\
			\dses{\airrr{\mu_{r,s}}^+}{\arelr{\mu_{r,s}}{q_{r,s}}}{\conj \airrr{\mu_{u-r,v-s}}^+}.
		\end{gather}
	\end{subequations}
	Here, $\conj$ and $\parrev$ denote the conjugation and parity-reversal functors, respectively (see \cref{sec:rhwsl2,sec:swosp,sec:rhwosp}).
\end{mtheorem}

\subsection{Outline} \label{sec:outline}

We begin, in \cref{sec:relaxed}, by recalling the definition of relaxed \hwms{} over an affine Kac-Moody superalgebra $\ag$ and introducing the module category in which we shall work.  We then specialise (\cref{sec:relaxedsl2}) to $\ag = \asltwo$, discussing the simple and certain carefully chosen non-simple weight $\sltwo$-modules, before inducing to obtain the relaxed $\asltwo$-modules of interest.

The study of the characters of these modules commences in \cref{sec:charsl2}.  First, the notion of a string function is recalled.  We then introduce relaxed coherent families and define a variant of the Shapovalov form on them.  We prove a key result about such forms (\cref{thm:limstrfnsexist}) which then allows us to compute the string functions of each relaxed $\asltwo$-module in \cref{sec:simpstrsl2}.  The structure of the non-simple relaxed modules is also discussed in \cref{sec:nsimpstrsl2} where we present an extended example to illustrate that this question is decidedly non-trivial in general.  \cref{sec:voasl2} then determines structures and computes characters explicitly when the relaxed $\asltwo$-module defines a module over the simple \voa{} $\slvoa{\kk}$, for general admissible levels $\kk$.

The remainder of the article studies the case $\ag = \aosp$.  There are many similarities with the $\asltwo$ case, with the main difference being the need to study a twisted (Ramond) sector in addition to the usual (\ns{}) sector.  \cref{sec:relaxedosp12} deals with the simple and non-simple $\osp$-modules and their inductions to relaxed $\aosp$-modules (\ns{} and Ramond), while \cref{sec:charosp} outlines the minor differences required to compute the string functions of the relaxed $\aosp$-modules.  The application to module characters and structures for the simple admissible-level \svoa{} $\ospvoa{\kk}$ appears in \cref{sec:voaosp}.  We conclude with \cref{app:vermasfn} in which string functions are studied for Verma modules over $\asltwo$ and $\aosp$ in order to simplify the character calculations in \cref{sec:voasl2,sec:voaosp}.

\section*{Acknowledgements}

We thank Dra\v{z}en Adamovi\'{c}, Tomoyuki Arakawa, Thomas Creutzig, Tianshu Liu and Simon Wood for useful discussions as well as their encouragement.  We also thank Will Stewart for pointing out a small error in a previous version and Ryo Sato for clarifying for us the relation between his work and relaxed $\AKMA{sl}{2}$-characters.
KK's research is supported by the Australian Research Council Discovery Project DP160101520.
DR's research is supported by the Australian Research Council Discovery Project DP160101520 and the Australian Research Council Centre of Excellence for Mathematical and Statistical Frontiers CE140100049.

\section{Relaxed \hwms{}} \label{sec:relaxed}

We recall here the relaxed \hwms{} introduced in \cite{FeiEqu98}, for $\asltwo$, and in \cite{RidRel15} for untwisted affine Kac-Moody algebras (actually, the setting in the latter paper covers relaxed modules for general conformally graded Lie superalgebras).  Given a simple Lie algebra $\g$ with a fixed choice of Cartan subalgebra $\h$, form the associated untwisted affine Kac-Moody algebra
\begin{equation}
	\ag = \g \otimes \CC[t,t^{-1}] \oplus \CC K \oplus \CC L_0,
\end{equation}
where $K$ is central and $L_0$ acts on $x_n \equiv x \otimes t^n$, $x \in \g$ and $n \in \ZZ$, as a derivation:  $\comm{L_0}{x_n} = -n x_n$.  Let $\ah = \h \oplus \CC K \oplus \CC L_0$.  We make the following definitions.
\begin{definition}
	\leavevmode
	\begin{itemize}
		\item The \emph{relaxed triangular decomposition} of an untwisted affine Kac-Moody algebra $\ag$ is
		\begin{equation}
			\ag = \ag^< \oplus \ag^0 \oplus \ag^> = \ag^< \oplus \ag^{\ge},
		\end{equation}
		where $\ag^{\ge} = \ag^0 \oplus \ag^>$, $\ag^<$ ($\ag^>$) is the subalgebra of $\ag$ consisting of the $x_n$ with $x \in \g$ and $n<0$ ($n>0$), and $\ag^0$ is the subalgebra spanned by $K$, $L_0$ and the $x_0$ with $x \in \g$.
		\item A \emph{relaxed \hwv{}} of $\ag$ is a simultaneous eigenvector of $\ah$ that is annihilated by $\ag^>$.
		\item A \emph{relaxed \hwm{}} of $\ag$ is a $\ag$-module that is generated by a single relaxed \hwv{}.
		\item A \emph{relaxed Verma module} of $\ag$ is a $\ag$-module isomorphic to $\rel{\Mod{M}} = \envalg{\ag} \otimes_{\ag^{\ge}} \Mod{M}$, where $\Mod{M}$ is some weight $\ag^0$-module on which $K$ and $L_0$ act as multiplication by some $\kk$ and $\Delta$ in $\CC$, respectively, extended to a $\ag^{\ge}$-module by letting $\ag^>$ act as $0$.
		\item A \emph{ground state} of a $\ag$-module $\aMod{M}$ is a generalised $L_0$-eigenvector whose $L_0$-eigenvalue is minimal among those of $\aMod{M}$.
	\end{itemize}
\end{definition}
\noindent Here, $\envalg{\ag^{\bullet}}$ denotes the \uea{} of $\ag^{\bullet}$, where $\bullet$ may stand for $>$, $\ge$, $0$, $\le$, $<$ or nothing.  If $\bullet$ is $\ge$, $0$, $\le$ or nothing, then it will be convenient in what follows to also consider
\begin{equation} \label{eq:defUk}
	\envalgk{\ag^{\bullet}} = \frac{\envalg{\comm{\ag}{\ag}} \cap \envalg{\ag^{\bullet}}}{\ideal{K - \kk \, \wun}}.
\end{equation}
This construction serves to remove $L_0$ as a generator and identify $K$ with a scalar multiple of the unit $\wun$ of $\envalg{\ag}$.

As usual, every relaxed \hwm{} may be realised as a quotient of some relaxed Verma module.  However, the relaxed Verma module $\rel{\Mod{M}}$ need not be a relaxed \hwm{} in general.  It will be, of course, if $\Mod{M}$ is a simple $\ag^0$-module.  Obviously, a relaxed \hwv{} of minimal conformal weight is a ground state, but the converse is not true in general.

Just as \hwms{} are typically analysed in the context of the \bgg{} category $\categ{O}$, it is useful to discuss relaxed \hwms{} as objects in a larger category.
\begin{definition}
	For an untwisted affine Kac-Moody algebra $\ag$, the associated \emph{relaxed category} $\categ{R}$ has, for objects, the $\ag$-modules $\aMod{M}$ satisfying the following conditions:
	\begin{itemize}
		\item $\aMod{M}$ is finitely generated.
		\item The action of $\h \oplus \CC K \subset \ah \subset \ag^0$ on $\aMod{M}$ is semisimple and the generalised simultaneous eigenspaces of the action of $\ah$ (its \emph{weight spaces}) are all finite-dimensional.
		\item The action of $\ag^>$ on $\aMod{M}$ is locally nilpotent: $\dim \brac*{\envalg{\ag^>} \cdot v} < \infty$ for all $v \in \aMod{M}$.
	\end{itemize}
	The morphisms are $\ag$-module homomorphisms, as usual.
\end{definition}
A relaxed \hwm{} belongs to $\categ{R}$ if and only if it has finite-dimensional weight spaces.  The same is true for a relaxed Verma module $\rel{\Mod{M}}$ and for this, it is sufficient that $\Mod{M}$ has finite-dimensional weight spaces (with respect to $\h$).  Moreover, every non-zero module in $\categ{R}$ has a relaxed \hwv{}.  It follows that the simple objects of $\categ{R}$ are relaxed \hwms{}.

\begin{remark}
	\leavevmode
	\begin{itemize}
		\item All this generalises to the affine Kac-Moody superalgebras corresponding to $\g$ being simple, basic and classical, as long as one respects the $\ZZ_2$-grading by parity throughout.
		\item For convenience, we shall understand throughout that the definition of \emph{weight module} always includes the requirement that its weight spaces are finite-dimensional.  When $\g$ is a Lie superalgebra, we shall also insist that weight modules are $\ZZ_2$-graded by parity.
		\item We do not insist that $L_0$ acts semisimply on modules in $\categ{R}$ because we would like to be able to accommodate non-semisimple actions when $\g$ is a Lie superalgebra like $\SLSA{sl}{2}{1}$.
	\end{itemize}
\end{remark}

\section{Relaxed \hw{} $\asltwo$-modules} \label{sec:relaxedsl2}

This \lcnamecref{sec:relaxedsl2} introduces the relaxed \hwms{} over $\asltwo$ that we are interested in.  We first recall the classification of simple weight modules over $\sltwo$, discussing the less familiar, but far more numerous, \emph{dense} modules in detail.  Certain non-simple dense $\sltwo$-modules are also introduced for later use.  Finally, we induce to obtain relaxed Verma $\asltwo$-modules and their (generically) simple quotients.

\subsection{Simple weight $\sltwo$-modules} \label{sec:swsl2}

We recall the classification of simple weight $\sltwo$-modules, recalling that we assume that weight modules have finite-dimensional weight spaces.  For this, we fix a basis $\set{e,h,f}$ such that
\begin{equation} \label{eq:commsl2}
	\comm{h}{e} = 2e, \quad \comm{e}{f} = h, \quad \comm{h}{f} = -2f,
\end{equation}
choose the Cartan subalgebra to be $\h = \CC h$, and normalise the quadratic Casimir in $\envalg{\sltwo}$ to be
\begin{equation}
	\cas = \frac{1}{2} h^2 + ef + fe.
\end{equation}
In this basis, the (rescaled) Killing form has non-zero entries
\begin{equation} \label{eq:killingsl2}
	\killing{h}{h} = 2, \quad \killing{e}{f} = \killing{f}{e} = 1.
\end{equation}
The bilinear form induced from the Killing form on $\h^*$ will be denoted by $\bilin{\cdot}{\cdot}$.  The rescaling normalises this form so that $\norm{\sroot}^2 = \bilin{\sroot}{\sroot} = 2$.

Let $\fwt \in \h^*$, $\sroot = 2 \fwt$ and $\wvec = \fwt$ denote the fundamental weight, the simple root and the Weyl vector of $\sltwo$, respectively.  Let $\wlat = \ZZ \fwt$ and $\rlat = \ZZ \sroot$ denote the weight and root lattices of $\sltwo$, respectively, while $\wlat_{\ge} = \ZZ_{\ge 0} \fwt$ denotes the dominant integral weights.  Finally, we introduce the following useful family of subsets of $\h^* / \rlat$, parametrised by $q \in \CC$:
\begin{equation} \label{eq:defLambda}
	\nsimpset{q} = \set*{[\lambda] \in \h^* / \rlat \st \bilin{\mu}{\mu + 2 \wvec} = q\ \text{for some}\ \mu \in [\lambda]}.
\end{equation}
The classification of simple weight $\sltwo$-modules is now succinctly stated as follows.

\begin{proposition}[see \protect{\cite[Thm.~3.32]{MazLec10}}] \label{prop:sl2simples}
	Every simple weight $\sltwo$-module (with finite-dimensional weight spaces) is isomorphic to precisely one member of one of the following families:
	\begin{enumerate}
		\item The finite-dimensional modules $\ffin{\mu}$ with highest weight $\mu \in \wlat_{\ge}$ and lowest weight $-\mu$.
		\item The \hw{} Verma modules $\fver{\mu}^+$ with highest weight $\mu \notin \wlat_{\ge}$.
		\item The \lw{} Verma modules $\fver{\mu}^-$ with lowest weight $\mu \notin -\wlat_{\ge}$.
		\item The dense modules $\fden{[\lambda]}{q}$ with weight support $[\lambda] \in \h^* / \rlat$ and $\cas$-eigenvalue $q \in \CC$ satisfying $[\lambda] \notin \nsimpset{q}$.
	\end{enumerate}
	All of these modules have one-dimensional weight spaces.
\end{proposition}

We recall that a dense module is one whose weight support is a translation of $\rlat$.  Whenever it will not cause confusion, we shall drop the brackets distinguishing $\lambda \in \h^*$ from its coset $[\lambda] \in \h^* / \rlat$, especially with regard to notation for dense modules: thus, $\fden{[\lambda]}{q} \equiv \fden{\lambda}{q}$.  Note that $\ffin{\mu}$, $\mu \in \wlat_{\ge}$, is left invariant by the functor induced from the Weyl reflection of $\sltwo$, while it exchanges $\fver{\mu}^+$ with $\fver{-\mu}^-$ and $\fden{\lambda}{q}$ with $\fden{-\lambda}{q}$, for $\lambda \notin \nsimpset{q}$ and $\mu \notin \wlat_{\ge}$.

\subsection{Non-simple dense $\sltwo$-modules} \label{sec:densesl2}

Fix $q \in \CC$ and consider the family of simple dense $\sltwo$-modules $\fden{\lambda}{q}$, $\lambda \notin \nsimpset{q}$, given in \cref{prop:sl2simples}.  It is clear that $f \in \sltwo$ acts injectively on each of these modules, as does $e$.  It follows that we may choose basis vectors $v_{\mu}$, $\mu \in \lambda$, of $\fden{\lambda}{q}$ so that the $\sltwo$-action on $\fden{\lambda}{q}$ is given by
\begin{equation} \label{eq:polyaction}
	e v_{\mu} = \gamma_{\mu} v_{\mu + \sroot}, \quad
	h v_{\mu} = \bilin{\mu}{\sroot} v_{\mu}, \quad
	f v_{\mu} = v_{\mu - \sroot}, \qquad
	\gamma_{\mu} = \frac{1}{2} \sqbrac[\big]{q - \bilin{\mu}{\mu + 2 \wvec}}.
\end{equation}
The key observation is that this action is polynomial in $\mu \in \h^*$.  To complete this family of dense $\sltwo$-modules, we shall choose a non-simple dense $\sltwo$-module, also denoted by $\fden{\lambda}{q}$, to fill each ``gap'' corresponding to the $\lambda \in \nsimpset{q}$.  This will be done by requiring that $f$ continues to act injectively.  It then follows that \eqref{eq:polyaction} will also hold for the non-simple $\fden{\lambda}{q}$.

To construct these non-simple modules, we recall that dense $\sltwo$-modules are easily obtained by inducing the simple modules of the centraliser of $\h$ in $\envalg{\sltwo}$.  Using the \pbw{} theorem, it is easy to see that this centraliser is $\CC[h,\cas]$.  Let $v$ denote a spanning vector of a (necessarily one-dimensional) simple $\CC[h,\cas]$-module, so that $hv = \lambda(h) v$ and $\cas v = qv$, for some $\lambda \in \h^*$ and some $q \in \CC$.  Then, a basis of the (obviously dense) induced $\sltwo$-module is $\set{v, e^n v, f^n v \st n \in \ZZ_{>0}}$.  Moreover, this module will be simple if and only if no $e^n v$ is a \lwv{} and no $f^n v$ is a \hwv{}, leading to the condition $[\lambda] \notin \nsimpset{q}$ stated in \cref{prop:sl2simples}.

If, however, we choose $\lambda \in \h^*$ such that $[\lambda] \in \nsimpset{q}$, then the induced $\sltwo$-module will be dense and indecomposable, but not simple.  The solutions in $\h^*$ of $\bilin{\mu}{\mu + 2 \wvec} = q$ have the form
\begin{equation} \label{eq:sols}
	\mu = -\wvec \pm \sqrt{1+2q} \, \fwt
\end{equation}
and are therefore distinct unless $q = -\norm{\wvec}^2 = -\frac{1}{2}$.  If there is precisely one such solution $\mu$ in $[\lambda] \in \h^* / \rlat$, meaning that $\sqrt{1+2q} \notin \ZZ \setminus \set{0}$, then the structure of the induced module depends only on whether $\lambda \le \mu$ or $\lambda > \mu$ (where the ordering is by the real part of the Dynkin label).  We choose $\fden{\lambda}{q} = \fden{[\lambda]}{q}$ to be the induced module obtained when $\lambda > \mu$.  Then, $\fden{\lambda}{q}$ has no \lwvs{} and so $f$ acts injectively, as desired, although $e$ does not.

If there are instead two (distinct) solutions \eqref{eq:sols} in $[\lambda]$, which requires that $\sqrt{1+2q} \in \ZZ \setminus \set{0}$, then let $\mu$ denote the maximal one (with respect to the ordering used above).  We have, therefore, $\mu \in \wlat_{\ge}$.  There are now three different possible structures for the induced $\sltwo$-modules according as to whether $\lambda > \mu$, $\lambda < -\mu$ or $-\mu \le \lambda \le \mu$.  We again choose $\fden{\lambda}{q} = \fden{[\lambda]}{q}$ to be the induced module obtained when $\lambda > \mu$ so that $f$ acts injectively.

For fixed $q \in \CC$, the number $\abs*{\nsimpset{q}}$ of (isomorphism classes of) non-simple $\fden{\lambda}{q}$ is therefore $1$ if $\sqrt{1+2q} \in \ZZ$ and is $2$ otherwise.  We can characterise each of these non-simples through its unique composition series.  If $\sqrt{1+2q} \in \ZZ \setminus \set{0}$ and $\mu \in \wlat_{\ge}$ is the maximal solution of $\bilin{\mu}{\mu + 2 \wvec} = q$, so that $\lambda = [\mu]$, then the composition series is
\begin{equation}
	0 \subset \fver{-\mu-\sroot}^+ \subset \fver{\mu}^+ \subset \fden{\lambda}{q}
\end{equation}
and its composition factors are $\fver{-\mu-\sroot}^+$, $\ffin{\mu}$ and $\fver{\mu+\sroot}^-$.  If $\sqrt{1+2q} \notin \ZZ \setminus \set{0}$ and $\mu$ is any solution of $\bilin{\mu}{\mu + 2 \wvec} = q$, then the composition series for $\lambda = [\mu]$ is instead
\begin{equation}
	0 \subset \fver{\mu}^+ \subset \fden{\lambda}{q}
\end{equation}
and the composition factors are $\fver{\mu}^+$ and $\fver{\mu+\sroot}^-$.

\begin{example}[$\sqrt{1+2q} \notin \ZZ$]
	Suppose we choose $q = -\frac{3}{8}$.  Then, $\bilin{\mu}{\mu + 2 \wvec} = q$ if and only if $\mu = -\frac{1}{2} \fwt$ or $-\frac{3}{2} \fwt$.  As the difference of these solutions is not in $\rlat$, it follows that $\fden{\lambda}{-3/8}$ is simple for all but two cosets $\lambda \in \h^* / \rlat$, one for each solution.  In other words, $\nsimpset{-\frac{3}{8}} = \set*{[-\frac{1}{2} \fwt], [-\frac{3}{2} \fwt]}$.  The corresponding non-simple dense $\sltwo$-modules are indecomposable with two composition factors each.  Moreover, they are completely characterised by the following short exact sequences:
	\begin{equation}
		\dses{\fver{-\fwt/2}^+}{\fden{-\fwt/2}{-3/8}}{\fver{3\fwt/2}^-}, \qquad
		\dses{\fver{-3\fwt/2}^+}{\fden{\fwt/2}{-3/8}}{\fver{\fwt/2}^-}.
	\end{equation}
\end{example}

\begin{example}[$\sqrt{1+2q} \in \ZZ \setminus \set{0}$]
	By way of contrast, taking $q=0$ yields $\mu = 0$ and $-2 \fwt$ as the solutions of $\bilin{\mu}{\mu + 2 \wvec} = q$.  The difference of these solutions does lie in $\rlat$, hence $\fden{\lambda}{0}$ is simple for all cosets except $\lambda \in \nsimpset{0} = \set*{[0]}$.  This exception is indecomposable, with three composition factors, and is characterised by the following short exact sequence:
	\begin{equation}
		\dses{\fver{0}^+}{\fden{0}{0}}{\fver{2 \fwt}^-}.
	\end{equation}
	Note that the Verma module $\fver{0}^+$ is not simple, having $\fver{-2 \fwt}^+$ as a simple proper submodule.
\end{example}

\begin{example}[$\sqrt{1+2q}=0$]
	The last type of example corresponds to $q=-\frac{1}{2}$, for which the only solution of $\bilin{\mu}{\mu + 2 \wvec} = q$ is $\mu = -\wvec$.  $\fden{\lambda}{-1/2}$ is therefore simple unless $\lambda \in \nsimpset{-\frac{1}{2}} = \set*{[\wvec]}$.  The non-simple dense module has two composition factors and is characterised by the following short exact sequence:
	\begin{equation}
		\dses{\fver{-\wvec}^+}{\fden{\wvec}{-1/2}}{\fver{\wvec}^-}.
	\end{equation}
\end{example}

\subsection{Relaxed \hw{} $\asltwo$-modules} \label{sec:rhwsl2}

Each of the simple $\sltwo$-modules $\Mod{M}$ of \cref{prop:sl2simples}, and more generally any indecomposable weight $\sltwo$-module, may be induced to a unique relaxed Verma module $\rel{\Mod{M}}$ of $\asltwo$, once we fix the eigenvalue $\kk$ of $K$, called the \emph{level}, and the eigenvalue $\Delta$ of $L_0$, called the \emph{conformal weight}, on $\Mod{M}$.  It is clear that $\rel{\Mod{M}}$ is in category $\categ{R}$ and that its space of ground states is naturally isomorphic to $\Mod{M}$ as an $\sltwo$-module.  We shall not specify the level $\kk$ or conformal weight $\Delta$ explicitly in our module notation, assuming that it is understood in the given context.

If we take $\Mod{M}$ to be one of the $\fver{\mu}^+$, then induction results in a Verma module (with respect to the standard Borel subalgebra of $\asltwo$).  Starting with $\Mod{M} = \fver{\mu}^-$, we instead obtain Verma modules with respect to the Borel obtained from the standard one by applying the Weyl reflection of $\sltwo$.  We denote the results by $\aver{\mu}^+$ and $\aver{\mu}^-$, respectively.  Their respective simple quotients will be denoted by $\airr{\mu}^+$ and $\airr{\mu}^-$.  The functor (on $\sltwo$-modules) induced from the Weyl reflection lifts to a functor on $\asltwo$-modules called \emph{conjugation}.  We shall denote this conjugation functor by $\conj$ so that $\conj \aver{\mu}^+ \cong \aver{-\mu}^-$ and $\conj \airr{\mu}^+ \cong \airr{-\mu}^-$.

If we instead take $\Mod{M} = \ffin{\mu}$, so $\mu \in \wlat_{\ge}$, then we arrive at a proper quotient of both $\aver{\mu}^+$ and $\aver{-\mu}^-$ which we shall denote by $\afin{\mu}$.  This is actually a parabolic Verma module (with respect to the parabolic subalgebra $\asltwo^{\ge}$) and its simple quotient will be denoted by $\airr{\mu}$.  Both $\aver{\mu}$ and $\airr{\mu}$ are self-conjugate.  We note that all of the relaxed Verma modules $\aver{\mu}^+$, $\aver{\mu}^-$ and $\afin{\mu}$, as well as their simple quotients $\airr{\mu}^+$, $\airr{\mu}^-$ and $\airr{\mu}$, are \hwms{} with respect to the standard or the Weyl-reflected Borel subalgebra of $\asltwo$.

The most interesting case is thus that of the relaxed Verma modules $\arver{\lambda}{q}$ that are induced from the dense $\sltwo$-modules $\fden{\lambda}{q}$.  These are not \hw{} with respect to any Borel.  Let $\arint{\lambda}{q}$ denote the sum of the submodules of $\arver{\lambda}{q}$ that have zero intersection with the space of ground states and let $\arel{\lambda}{q} = \arver{\lambda}{q} \big/ \arint{\lambda}{q}$.  The $\arel{\lambda}{q}$ are likewise not \hw{} with respect to any Borel.  However, they are simple for all $\lambda \notin \nsimpset{q}$ as $\arint{\lambda}{q}$ then coincides with the maximal proper submodule $\armax{\lambda}{q}$ of $\arver{\lambda}{q}$ (which is unique because $\arver{\lambda}{q}$ is cyclic).  We shall identify the space of ground states of both $\arver{\lambda}{q}$ and $\arel{\lambda}{q}$ with $\fden{\lambda}{q} = \bigoplus_{\mu \in \lambda} \CC v_{\mu}$, so that the action of the zero modes $e_0$, $h_0$ and $f_0$ on the ground states is given by \eqref{eq:polyaction}.  We remark that $\conj \arver{\lambda}{q} \cong \arver{-\lambda}{q}$ and $\conj \arel{\lambda}{q} \cong \arel{-\lambda}{q}$, when $\lambda \notin \nsimpset{q}$, but that these isomorphisms fail for $\lambda \in \nsimpset{q}$.

Our aim in this paper is to rigorously determine the characters of the simple $\arel{\lambda}{q}$.  The key to this computation is to consider the result when $\lambda \in \nsimpset{q}$, that is when these relaxed \hwms{} are not simple.

\section{Relaxed $\asltwo$-modules and their string functions} \label{sec:charsl2}

In this \lcnamecref{sec:charsl2}, we study the string functions of the relaxed \hw{} $\asltwo$-modules $\arel{\lambda}{q}$.  The aim is to compute them in terms of the ``limiting'' string functions of certain associated simple \hwms{}.  This will be achieved by introducing affine versions of Mathieu's coherent families \cite{MatCla00} and studying analogues of Shapovalov forms on them.

\subsection{String functions} \label{sec:stringsl2}

Recall that the \emph{character} of a level-$\kk$ weight module $\aMod{M}$ over $\asltwo$ is given by
\begin{equation} \label{eq:defchar}
	\fch{\aMod{M}}{\zz;\qq} = \traceover{\aMod{M}} \zz^{h_0} \qq^{L_0} = \sum_{\mu \in \h^*, n \in \CC} \dim \aMod{M}(\mu,n) \, \zz^{\mu} \qq^n,
\end{equation}
where $\qq$ and $\zz$ are indeterminates and $\aMod{M}(\mu,n)$ denotes the weight space of $\aMod{M}$ with $\sltwo$-weight $\mu \in \h^*$ and conformal weight $n$.  The \emph{string function} $\tsfn{\mu}{\aMod{M}}$, $\mu \in \h^*$, of $\aMod{M}$ is then the coefficient of $\zz^{\mu}$ in the character:
\begin{equation} \label{eq:defsfn}
	\fsfn{\mu}{\aMod{M}}{\qq} = \sum_{n \in \CC} \dim \aMod{M}(\mu,n) \, \qq^n.
\end{equation}
We make the following definition.
\begin{definition}
	A level-$\kk$ weight module $\aMod{M}$ is said to be \emph{stringy} if its non-zero string functions $\tsfn{\mu}{\aMod{M}}$ all coincide.
\end{definition}
\noindent This means, in particular, that the multiplicities $\dim \aMod{M}(\mu,n)$ of the weights of $\aMod{M}$ are independent of $\mu$, provided only that $\mu$ is in the weight support of $\aMod{M}$.

\begin{example}
	Straightforward examples of stringy $\asltwo$-modules are provided by the level-$\kk$ relaxed Verma modules $\arver{\lambda}{q}$, where $\lambda \in \h^* / \rlat$ and $q \in \CC$ (see \cref{sec:rhwsl2}).  Indeed, their characters are easily computed:
	\begin{equation}
		\fch{\arver{\lambda}{q}}{\zz;\qq} = \frac{\qq^{\Delta + 1/8}}{\eta(\qq)^3} \sum_{\mu \in \lambda} \zz^{\mu} \qquad \implies \qquad
		\fsfn{\mu}{\arver{\lambda}{q}}{\qq} =
		\begin{dcases*}
			\frac{\qq^{\Delta + 1/8}}{\eta(\qq)^3}, & if $\mu \in \lambda$, \\
			0, & otherwise.
		\end{dcases*}
	\end{equation}
	Here, $\eta(\qq) = \qq^{1/24} \prod_{i=1}^{\infty} (1-\qq^i)$ is Dedekind's eta function.
\end{example}

\begin{remark}
	In applications to \voas{} and \cft{}, it is common to normalise characters (and thus string functions) by multiplying by $\qq^{-\cc / 24}$, where $\cc = \frac{3\kk}{\kk+\dcox}$ is the central charge of the theory (and $\kk \neq -\dcox = -2$).  Moreover, in this case, the Sugawara construction also fixes $\Delta$ as a function of $q$ and $\kk$.  We shall make this adjustment when applying our results to relaxed modules over the affine \voa{} $\slvoa{\kk}$ in \cref{sec:voasl2} below.
\end{remark}

We refer to series like string functions as generalised formal power series.  There is a useful partial ordering on generalised formal power series in $\qq$ defined by
\begin{equation} \label{eq:gfps<}
	\sum_{n \in \CC} a_n \qq^n \le \sum_{n \in \CC} b_n \qq^n \qquad \text{if} \qquad a_n \le b_n \qquad \text{for each $n \in \CC$.}
\end{equation}
If $\brac[\big]{S_m(\qq)}_{m \in \ZZ}$ is a sequence of generalised formal power series in $\qq$, then we say that this sequence converges to another generalised formal power series $S(\qq)$ if the coefficients in their expansions do.  More precisely, if we have
\begin{equation}
	S_m(\qq) = \sum_{n \in \CC} a_{m,n} \qq^n \qquad \text{and} \qquad S(\qq) = \sum_{n \in \CC} a_n \qq^n,
\end{equation}
then we shall write
\begin{equation} \label{eq:gfpslim}
	\lim_{m \to \pm \infty} S_m(\qq) = S(\qq) \qquad \text{if} \qquad
	\lim_{m \to \pm \infty} a_{m,n} = a_n \qquad \text{for each $n \in \CC$.}
\end{equation}
In what follows, we shall find it convenient to denote these limiting generalised formal power series by $S_{\pm\infty}(\qq)$.  In particular, when $\aMod{M}$ is indecomposable, so its weight support is a single coset $[\mu] \in \h^* / \rlat$, we shall define limiting string functions by
\begin{equation}
	\fsfn{\pm \infty}{\aMod{M}}{\qq} = \lim_{m \to \pm \infty} \fsfn{\mu + m \sroot}{\aMod{M}}{\qq},
\end{equation}
whenever the \rhs{} exists.

\subsection{Coherent families and Shapovalov forms} \label{sec:moretools}

Our first aim is to prove that the relaxed \hw{} $\asltwo$-modules $\arel{\lambda}{q}$ are stringy.  For this, we shall employ two key tools.  The first is Mathieu's notion of a coherent family \cite{MatCla00}.  This is a (highly reducible) module that is parametrised by its central character: for $\sltwo$, this is just the eigenvalue $q$ of the quadratic Casimir.  Although there is always more than one coherent family for each central character, the conventions introduced above (to facilitate the present application) pick one out uniquely.  We shall lift these preferred coherent families to relaxed coherent families over $\asltwo$.  These $\asltwo$-modules will be crucial for establishing the stringiness of the $\arel{\lambda}{q}$.

The coherent families that we shall use for $\sltwo$ are the direct sums
\begin{equation} \label{eq:cohfamsl2}
	\fcoh{q} = \bigoplus_{\lambda \in \h^* / \rlat} \fden{\lambda}{q}, \quad q \in \CC.
\end{equation}
Each of these has a one-dimensional weight space for every weight $\mu \in \h^*$.  Recall that we chose the $\fden{\lambda}{q}$ in \cref{sec:densesl2} so that the action of $f$ on each $\fden{\lambda}{q}$ would be injective.  The vectors $v_{\mu}$, now with $\mu \in \h^*$, therefore define a basis of $\fcoh{q}$ on which the $\sltwo$-action is again given by \eqref{eq:polyaction}.  We emphasise that this action is manifestly polynomial in $\mu$.

We introduce two affine versions of the $\sltwo$ coherent families of \eqref{eq:cohfamsl2}.  These \emph{relaxed coherent families} are $\asltwo$-modules and we have one version that decomposes into relaxed Verma modules and one into their generically simple quotients:
\begin{equation} \label{eq:defaffcohfams}
	\avcoh{q} = \bigoplus_{\lambda \in \h^* / \rlat} \arver{\lambda}{q}, \qquad
	\acoh{q} = \bigoplus_{\lambda \in \h^* / \rlat} \arel{\lambda}{q}.
\end{equation}
These modules do not share the property of having one-dimensional weight spaces (with respect to the Cartan subalgebra $\ah$ of $\asltwo$).  However, they do admit a polynomial action of $\asltwo$ and so provide a useful setting for comparing the properties of their summands.

The second tool that we shall need is an analogue of the Shapovalov form on the relaxed coherent families $\avcoh{q}$.  To construct this, we first construct such forms on the relaxed Verma modules $\arver{\lambda}{q}$.  Our definition depends on two choices: a cyclic generator of $\arver{\lambda}{q}$ and an adjoint (linear involutive antiautomorphism) of $\envalg{\asltwo}$.  For the generator, we shall choose a ground state $v_{\nu}$, $\nu \in \lambda$.  This may be chosen arbitrarily when $\lambda \notin \nsimpset{q}$.  When $\lambda \in \nsimpset{q}$, we must choose a $v_{\nu}$ with $\nu > \mu$, where $\mu$ is the maximal solution in $\lambda$ of $\bilin{\mu}{\mu+2\wvec} = q$.  For the adjoint, we take the extension to $\envalg{\asltwo}$ of the compact adjoint of $\asltwo$:
\begin{equation}
	e_n^{\dag} = f_{-n}, \qquad h_n^{\dag} = h_{-n}, \qquad f_n^{\dag} = e_{-n}, \qquad K^{\dag} = K, \qquad L_0^{\dag} = L_0.
\end{equation}
Given these choices, recalling \cref{eq:defUk} and noting that $v_{\nu}$ is a simultaneous eigenvector of $K$ and $L_0$, we define a contravariant bilinear form $\inner{\cdot}{\cdot}_{\nu}$ on $\arver{\lambda}{q}$ by
\begin{equation} \label{eq:defshap}
	\inner{v_{\nu}}{v_{\nu}}_{\nu} = 1 \quad \text{and} \quad \inner{Uv_{\nu}}{Vv_{\nu}}_{\nu} = \inner{v_{\nu}}{U^{\dag}Vv_{\nu}}_{\nu}, \qquad \text{for all}\ U,V \in \envalgk{\asltwo}.
\end{equation}
We call it a \emph{Shapovalov form} on $\arver{\lambda}{q}$.  Note that the kernel of such a Shapovalov form on $\arver{\lambda}{q}$ coincides with the maximal proper submodule $\armax{\lambda}{q}$ and that this does not depend on the choices made during the construction.

To check that this form is well defined, note that as $h_0$ and $L_0$ are both self-adjoint, their simultaneous eigenspaces are orthogonal with respect to $\inner{\cdot}{\cdot}_{\nu}$.  Taking a \pbw{} ordering such that mode indices increase to the right, we see that $\inner{Uv_{\nu}}{Vv_{\nu}}_{\nu}$ vanishes if $U^{\dag}V$ belongs to the span $Z$ of the ordered monomials that either involve a non-zero mode index or have a non-zero $\sltwo$-weight.  It follows that the value of the form \eqref{eq:defshap} is entirely determined by the projection $\beta \colon \envalgk{\asltwo} \to \CC[h,\cas]$ whose kernel is $Z$:
\begin{equation} \label{eq:defshap2}
	\inner{Uv_{\nu}}{Vv_{\nu}}_{\nu} = \left. \beta(U^{\dag}V) \right\rvert_{h \mapsto \nu(h), Q \mapsto q}.
\end{equation}
Here, we have identified the image of $\beta$ with the centraliser of $\h$ in $\envalg{\sltwo}$ (\cref{sec:densesl2}).

Fix now $q \in \CC$.  For each $\lambda \in \h^* / \rlat$, choose a $\nu \in \lambda$ that defines a Shapovalov form $\inner{\cdot}{\cdot}_{\nu}$ on $\arver{\lambda}{q}$.  The direct sum
\begin{equation}
	\bigoplus_{\lambda \in \h^* / \rlat} \inner{\cdot}{\cdot}_{\nu}
\end{equation}
then defines a contravariant bilinear form, which we shall also refer to as a Shapovalov form, on the relaxed coherent family $\avcoh{q}$.  This construction clearly depends on the uncountably many choices for $\nu$, one for each $\lambda \in \h^* / \rlat$.  However, the kernel of this form is independent of these choices.  Note that this construction is equivalent to extending the chosen Shapovalov forms on the $\arver{\lambda}{q}$ to $\avcoh{q}$ by insisting that $v_{\xi}$ and $v_{\zeta}$ are orthogonal for all distinct $\xi, \zeta \in \h^*$ (consistent with $h_0$ being self-adjoint).

We are now almost ready for the key technical result, \cref{lem:rankconst} below.  First, however, recall that when $\lambda \notin \nsimpset{q}$, we have $\arint{\lambda}{q} = \armax{\lambda}{q}$.  When $\lambda \in \nsimpset{q}$, the following result will prove to be a useful substitute.
\begin{lemma} \label{lem:I=J}
	Suppose that $\lambda \in \nsimpset{q}$ and let $\mu$ be the maximal solution in $\lambda$ of $\bilin{\mu}{\mu+2\wvec} = q$.  Then,
	\begin{equation}
		\arint{\lambda}{q}(\mu+m\sroot,\Delta+n) = \armax{\lambda}{q}(\mu+m\sroot,\Delta+n),
	\end{equation}
	for all $m > n \in \ZZ_{\ge 0}$.
\end{lemma}
\begin{proof}
	As $\arint{\lambda}{q} \subseteq \armax{\lambda}{q}$ is clear, we suppose that $v \in \armax{\lambda}{q}(\mu+m\sroot,\Delta+n)$.  Because each ground state $v_{\nu}$, with $\nu > \mu$, generates $\arver{\lambda}{q}$, the submodule $\aMod{M}_v \subseteq \armax{\lambda}{q} \subset \arver{\lambda}{q}$ generated by $v$ has zero intersection with $\bigoplus_{\nu > \mu} \CC v_{\nu}$.  Assume that one of the other ground states $v_{\nu}$, $\nu \le \mu$, belongs to $\aMod{M}_v$.  Applying \pbw{} basis elements (with indices increasing to the right) to $v$ now shows that so must $v_{\mu + (m-n) \sroot}$, a contradiction since $m>n$.  Thus, $\aMod{M}_v$ has zero intersection with the space of ground states $\bigoplus_{\nu \in \lambda} \CC v_{\nu}$ and so $v \in \arint{\lambda}{q}$.
\end{proof}

Fix $n \in \ZZ_{\ge 0}$ and define $P_n$ to be the set of all \pbw{} monomials of $\envalgk{\asltwo^{\le 0}}$, ordered so that mode indices increase to the right, that satisfy the following conditions:
\begin{itemize}
	\item The $\sltwo$-weight ($\ad(h_0)$-eigenvalue) is $-n \sroot$.
	\item The conformal grade (the negative of the sum of the mode indices) is $n$.
	\item The exponents of $e_0$ and $h_0$ are zero.
\end{itemize}
There are clearly only finitely many such monomials.  A basis for the weight space $\avcoh{q}(\nu,\Delta+n)$ is then given by the $Uv_{\nu+n\sroot}$ with $U \in P_n$.

Choose a Shapovalov form on $\avcoh{q}$.  Then, for each $\nu \in \h^*$ and $n \in \ZZ_{\ge 0}$, we define the \emph{Shapovalov matrix} for $\avcoh{q}(\nu,\Delta+n)$ to be the $\abs*{P_n} \times \abs*{P_n}$ matrix
\begin{equation}
	A_{\nu;n} = \brac[\big]{\inner{Uv_{\nu+n\sroot}}{Vv_{\nu+n\sroot}}_{\nu}}_{U,V \in P_n}.
\end{equation}
The kernel of this matrix is then the weight space $\armax{\lambda}{q}(\nu,\Delta+n)$.  If $\lambda \notin \nsimpset{q}$, then $\armax{\lambda}{q} = \arint{\lambda}{q}$, so the rank of $A_{\nu;n}$ is the dimension of $\arel{\lambda}{q}(\nu,\Delta+n)$.  This, in turn, is the coefficient of $\qq^{\Delta+n}$ in the string function $\tfsfn{\nu}{\arel{\lambda}{q}}{\qq}$.  If $\lambda \in \nsimpset{q}$, then \cref{lem:I=J} gives the same conclusion for all $\nu > \mu + n \sroot$.
\begin{lemma} \label{lem:rankconst}
	For each $n \in \ZZ_{\ge 0}$, the rank of the Shapovalov matrix $A_{\nu;n}$ is independent of $\nu \in \h^*$ for sufficiently large $\nu$.
\end{lemma}
\begin{proof}
	Fix $n$ and $q \in \CC$.  Then, the entries of $A_{\nu;n}$ are complex polynomials in $\nu(h) \in \CC$, by \eqref{eq:defshap2}.  Let $B_{\nu;n}$ denote its reduced row-echelon form over $\CC$.  If we instead treat $\nu$ as a formal indeterminate, writing $A_n(\nu)$ for the Shapovalov matrix in this case, then we may instead row-reduce over the field $\CC(\nu)$ of rational functions in $\nu$.  Let $B_n(\nu)$ denote the reduced row-echelon form, over $\CC(\nu)$, of $A_n(\nu)$.  Then, evaluating $\nu$ at $\nu(h) \in \CC$ gives $\left. B_n(\nu) \right\rvert_{\nu \mapsto \nu(h)} = B_{\nu;n}$, for all but finitely many $\nu(h)$ (because row-reduction gives only finitely many opportunities to divide by zero).  Similarly, each non-zero entry of $B_n(\nu)$ will evaluate to a non-zero entry of $B_{\nu;n}$ for all but finitely many $\nu(h)$.  As there are only finitely many entries, it follows that the number of non-zero rows of $B_n(\nu)$ and $B_{\nu;n}$ must agree for all but finitely many values of $\nu(h) \in \CC$.  This number for $B_n(\nu)$ is obviously independent of $\nu(h)$, so the \lcnamecref{lem:rankconst} follows.
\end{proof}
\begin{remark}
	The statement of the \lcnamecref{lem:rankconst} would also hold for $\nu$ sufficiently small (negative), except that our construction of Shapovalov forms required us, when $\lambda \in \nsimpset{q}$, to choose $\nu \in \lambda$ larger than the maximal solution $\mu$.
\end{remark}
\noindent This \lcnamecref{lem:rankconst} immediately implies our first result on limiting string functions.
\begin{theorem} \label{thm:limstrfnsexist}
	For given $q \in \CC$, the positive limiting string functions $\sfn{\infty}{\arel{\lambda}{q}}$ exist and are independent of $\lambda \in \h^* / \rlat$.
\end{theorem}

\subsection{Stringiness of the simple $\arel{\lambda}{q}$} \label{sec:simpstrsl2}

Recall that the $\arel{\lambda}{q}$ are simple when $\lambda \notin \nsimpset{q}$, that is when the space of ground states is simple (as an $\sltwo$-module).  Our aim here is to show that the simple $\arel{\lambda}{q}$ are stringy.  This uses the following lemmas, the first of which is an immediate application of the \pbw{} theorem for $\asltwo^<$.
\begin{lemma} \label{lem:intersection}
	If  $\mu\neq \nu$, then $\envalg{\asltwo^{<}}v_\mu\cap \envalg{\asltwo^{<}}v_\nu=0$ in $\arver{\lambda}{q}$.
\end{lemma}
\begin{lemma} \label{prop:injective}
	The action of $f_0$ on $\arver{\lambda}{q}$ is injective.
	If $\lambda \notin \nsimpset{q}$, then $e_0$ also acts injectively on $\arver{\lambda}{q}$.
\end{lemma}
\begin{proof}
	We only show the first assertion as the second may be proved in a~similar fashion, once we recall that the condition on $\lambda$ and $q$ implies that the ground states $v_{\mu}$ span a \emph{simple} $\sltwo$-module isomorphic to $\fden{\lambda}{q}$, hence that $e_0$ does not annihilate any of the $v_{\mu}$.

	Let $w$ be an~ arbitrary non-zero element of $\arver{\lambda}{q}$, so that $w$ has the form
	\begin{equation}
		w=\sum_{i=1}^\ell U_i v_{\lambda+n_i\sroot},
	\end{equation}
	for some $\ell\in \ZZ_{>0}$,  $U_1,\ldots,U_{\ell}\in \envalg{\asltwo^{<}} \setminus \set{0}$ and $n_1 < \dots < n_\ell\in\ZZ$.
	Since $\comm{f_0}{U_i} \in \envalg{\asltwo^{<}}$, for each $i$, and $f_0$ does not annihilate any of the $v_{\mu}$, we see that
	\begin{equation}
		f_0 w = \sum_{i=1}^\ell \brac*{U_i v_{\lambda+(n_i-1)\sroot} + \comm*{f_0}{U_i} v_{\lambda+n_i\sroot}}
		\in U_1 v_{\lambda+(n_1-1)\sroot} + \bigoplus_{m \ge n_1} \envalg{\asltwo^{<}}v_{\lambda+m\sroot}.
	\end{equation}
	As $U_1 \neq 0$, the term $U_1 v_{\lambda+(n_1-1)\sroot}$ is non-zero.  Moreover, it cannot be cancelled by any of the other terms, by \cref{lem:intersection}.  Thus, $f_0 w \neq 0$ as desired.
\end{proof}
\begin{lemma} \label{lem:inj=stringy}
	If $e_0$ and $f_0$ both act injectively on an indecomposable level-$\kk$ weight module $\aMod{M}$ of $\asltwo$, then $\aMod{M}$ is stringy.
\end{lemma}
\begin{proof}
	Recall that the weight spaces $\aMod{M}(\mu,n)$, for $\mu \in \h^*$ and $n \in \CC$, are always finite-dimensional, by definition.  As $e_0 \colon \aMod{M}(\mu,n) \to \aMod{M}(\mu+\sroot,n)$ is assumed to act injectively, we have $\dim \aMod{M}(\mu,n) \le \dim \aMod{M}(\mu+\sroot,n)$.  Similarly, $f_0 \colon \aMod{M}(\mu+\sroot,n) \to \aMod{M}(\mu,n)$ acting injectively implies that $\dim \aMod{M}(\mu,n) \ge \dim \aMod{M}(\mu+\sroot,n)$.  The stringiness of $\aMod{M}$ now follows because indecomposability implies that the $\aMod{M}(\mu,n)$ are zero unless $\mu$ belongs to a unique coset $\lambda \in \h^* / \rlat$.
\end{proof}
The desired stringiness result is now easy to prove.
\begin{theorem} \label{thm:simplestringy}
	Let $q \in \CC$ and $\lambda \notin \nsimpset{q}$.  Then, the simple relaxed \hwm{} $\arel{\lambda}{q}$ is stringy.
\end{theorem}
\begin{proof}
	As $e_0$ and $f_0$ both act injectively on the maximal proper submodule $\armax{\lambda}{q} \subset \arver{\lambda}{q}$, by \cref{prop:injective}, it follows that $\armax{\lambda}{q}$ is stringy, by \cref{lem:inj=stringy}.  But, $\arver{\lambda}{q}$ is stringy (\cref{sec:stringsl2}), so we conclude that $\arel{\lambda}{q}=\arver{\lambda}{q} \big/ \armax{\lambda}{q}$ is too.
\end{proof}

\subsection{Computing the string functions} \label{sec:compstrfnsl2}

\cref{thm:simplestringy} says that the simple $\arel{\lambda}{q}$ are stringy, but we do not yet have a means to actually compute their string functions.  For this, we shall combine this result with \cref{thm:limstrfnsexist}, concluding that the string functions of the simple $\arel{\lambda}{q}$ coincide with the positive limiting string function of the non-simple ones.  As we shall see, the latter are computable in principle.

\begin{lemma} \label{lem:simpquot}
	Let $\lambda \in \nsimpset{q}$ and take $\mu$ to be the maximal solution in $\lambda$ of $\bilin{\mu}{\mu+2\wvec} = q$.  Then, $\airr{\mu+\sroot}^-$ is the unique simple quotient of both $\arver{\lambda}{q}$ and $\arel{\lambda}{q}$.
\end{lemma}
\begin{proof}
	Recall from \cref{sec:densesl2} that $\fver{\mu+\sroot}^-$ is a quotient of $\fden{\lambda}{q}$.  As induction is a tensor functor, it is right-exact, hence $\aver{\mu+\sroot}^-$ is a quotient of $\arver{\lambda}{q}$.  It follows that the irreducible $\airr{\mu+\sroot}^-$ is also a quotient of $\arver{\lambda}{q}$, necessarily by the (unique) maximal proper submodule $\armax{\lambda}{q}$.  This establishes the statement for $\arver{\lambda}{q}$ and that for $\arel{\lambda}{q}$ is obtained by noting that
	\begin{equation}
		\frac{\arel{\lambda}{q}}{\armax{\lambda}{q} \big/ \arint{\lambda}{q}}
		\cong \frac{\arver{\lambda}{q} \big/ \arint{\lambda}{q}}{\armax{\lambda}{q} \big/ \arint{\lambda}{q}}
		\cong \frac{\arver{\lambda}{q}}{\armax{\lambda}{q}}
		\cong \airr{\mu+\sroot}^-,
	\end{equation}
	remembering that $\arel{\lambda}{q}$ is cyclic.
\end{proof}

\begin{proposition} \label{lem:nsimplimstrfn}
	The limiting string function of $\arel{\lambda}{q}$, $\lambda \in \nsimpset{q}$, is
	\begin{equation} \label{eq:limsfn}
		\fsfn{\infty}{\arel{\lambda}{q}}{\qq} = \fsfn{\infty}{\airr{\mu+\sroot}^-}{\qq},
	\end{equation}
	where $\mu$ is the maximal solution in $\lambda$ of $\bilin{\mu}{\mu+2\wvec} = q$.
\end{proposition}
\begin{proof}
	Choose non-negative integers $m$ and $n$ satisfying $m>n$.  Then, \cref{lem:I=J,lem:simpquot} give
	\begin{align}
		\dim \arel{\lambda}{q}(\mu + m \sroot,\Delta+n) &= \dim \arver{\lambda}{q}(\mu + m \sroot,\Delta+n) - \dim \arint{\lambda}{q}(\mu + m \sroot,\Delta+n) \\
		&= \dim \arver{\lambda}{q}(\mu + m \sroot,\Delta+n) - \dim \armax{\lambda}{q}(\mu + m \sroot,\Delta+n) = \dim \airr{\mu + \sroot}^-(\mu + m \sroot,\Delta+n) \notag
	\end{align}
	and the desired identity of limiting string functions follows.
\end{proof}
\begin{remark}
	Recall that $\bilin{\mu}{\mu+2\wvec} = q$ has two solutions $\mu_{\pm} \in \h^*$, given in \eqref{eq:sols}, that satisfy $\mu_+ + \mu_- = -\sroot$.  When $\sqrt{1+2q} \notin \ZZ$, the cosets $\lambda_+ = [\mu_+]$ and $\lambda_- = [\mu_-]$ are distinct elements of $\nsimpset{q}$, hence \eqref{eq:limsfn} applies to both.  We must therefore have
	\begin{equation} \label{eq:nice}
		\fsfn{\infty}{\airr{\mu_{\pm}+\sroot}^-}{\qq} = \fsfn{\infty}{\arel{\lambda_{\pm}}{q}}{\qq} = \fsfn{\infty}{\arel{\lambda_{\mp}}{q}}{\qq} = \fsfn{\infty}{\airr{\mu_{\mp}+\sroot}^-}{\qq} = \fsfn{\infty}{\airr{-\mu_{\pm}}^-}{\qq},
	\end{equation}
	by \cref{thm:limstrfnsexist}.
\end{remark}

Combining \cref{lem:nsimplimstrfn,eq:nice} with \cref{thm:limstrfnsexist,thm:simplestringy}, we now deduce the string functions of the simple $\arel{\lambda}{q}$.
\begin{theorem} \label{thm:simpstrfns}
	If $\sqrt{1+2q} \notin \ZZ$, then the non-zero string functions of the simple relaxed \hwms{} $\arel{\lambda}{q}$, $\lambda \notin \nsimpset{q}$, have the form
	\begin{equation} \label{eq:simlim=lim}
		\fsfn{\nu}{\arel{\lambda}{q}}{\qq} = \fsfn{\infty}{\airr{\mu+\sroot}^-}{\qq}, \qquad \text{for all $\nu \in \lambda$,}
	\end{equation}
	where $\mu$ is any solution of $\bilin{\lambda}{\lambda+2\wvec} = q$.  If $\sqrt{1+2q} \in \ZZ$, then the same is true when $\mu$ is the maximal such solution.
\end{theorem}
\begin{remark}
	The irreducible $\asltwo$-modules $\airr{\nu}^-$ and $\airr{-\nu}^+$ are related by the conjugation functor $\conj$.  It follows that the positive limiting string function of one must match the negative limiting string function of the other.  We may therefore replace the \rhs{} of \eqref{eq:simlim=lim} with the negative limiting string function $\tfsfn{-\infty}{\airr{-\mu-\sroot}^+}{\qq}$.  Moreover, when $\sqrt{1+2q} \notin \ZZ$, we may instead replace this by $\tfsfn{-\infty}{\airr{\mu}^+}{\qq}$, by \eqref{eq:nice}.
\end{remark}

While \cref{thm:limstrfnsexist} assures us that the limiting string functions of the simple \hw{} $\asltwo$-modules appearing on the \rhs{} of \eqref{eq:simlim=lim} actually exist, it is perhaps comforting and useful to see this directly.  One way to approach this is to note, as in \cref{prop:vermasfnsl2}, that these limiting string functions also exist for Verma modules over $\asltwo$.  Indeed, $f_0$ acts injectively on the ground states of $\airr{\nu}^+$, $\nu \notin \wlat_{\ge}$, so the argument used in the proof of \cref{prop:injective} shows that $f_0$ acts injectively on all of $\airr{\nu}^+$.  The string functions $\tsfn{\xi}{\airr{\nu}^+}$ therefore increase monotonically as $\xi \to -\infty$, while they are bounded above by the limiting string function of $\aver{\nu}^+$.

\subsection{Stringiness of the non-simple $\arel{\lambda}{q}$} \label{sec:nsimpstrsl2}

While our first main aim, to compute the characters of the simple $\arel{\lambda}{q}$, was essentially completed in \cref{thm:simpstrfns}, it is now straightforward to also establish the stringiness of the non-simple $\arel{\lambda}{q}$ and thereby determine their characters.  We shall also discuss the structure of these $\asltwo$-modules.

\begin{lemma} \label{lem:simplesubmod}
	Let $\lambda \in \nsimpset{q}$ and take $\mu$ to be the maximal solution in $\lambda$ of $\bilin{\mu}{\mu+2\wvec} = q$.  Then, $\arel{\lambda}{q}$ has a simple submodule isomorphic to $\airr{-\mu-\sroot}^+$, if $\sqrt{1+2q} \in \ZZ$, and to $\airr{\mu}^+$ otherwise.
\end{lemma}
\begin{proof}
	Recall from \cref{sec:densesl2} that $\fden{\lambda}{q}$ has a simple submodule isomorphic to $\fver{-\mu-\sroot}^+$, if $\sqrt{1+2q} \in \ZZ$, and to $\fver{\mu}^+$ otherwise.  Let us assume that $\sqrt{1+2q} \notin \ZZ$ for simplicity.  Then, upon inducing to $\arver{\lambda}{q}$, the ground state $v_{\mu}$ becomes a \hwv{} for $\asltwo$, hence it generates a copy of $\aver{\mu}^+$ (as $\envalg{\asltwo^<}$ and $f_0$ act freely).  Now, the maximal proper submodule $\aMod{M}$ of $\aver{\mu}^+$ has zero intersection with the space of ground states, hence $\aMod{M} \subset \arint{\lambda}{q}$.  Indeed, the space $\fver{\mu}^+$ of ground states of $\aver{\mu}^+$ is simple, since $\mu \notin \wlat_{\ge}$, and so $\aMod{M} = \aver{\mu}^+ \cap \arint{\lambda}{q}$.  Thus,
	\begin{equation}
		\airr{\mu}^+
		\cong \frac{\aver{\mu}^+}{\aMod{M}}
		\cong \frac{\aver{\mu}^+}{\aver{\mu}^+ \cap \arint{\lambda}{q}}
		\lira \frac{\arver{\lambda}{q}}{\arint{\lambda}{q}}
		\cong \arel{\lambda}{q},
	\end{equation}
	as required.  If $\sqrt{1+2q} \in \ZZ$, then the argument goes through with $-\mu-\sroot$ replacing $\mu$ throughout.
\end{proof}
\begin{remark}
	Note that for the special case $\sqrt{1+2q} = 0$, we have $\mu = -\wvec$ and thus $-\mu-\sroot$ and $\mu$ coincide.
\end{remark}
\begin{theorem} \label{thm:nsimpstrfns}
	If $\lambda \in \nsimpset{q}$, then, $\arel{\lambda}{q}$ is stringy and its non-zero string functions are given by \eqref{eq:simlim=lim}.
\end{theorem}
\begin{proof}
	Since $f_0$ acts injectively on $\arint{\lambda}{q} \subset \arver{\lambda}{q}$, by \cref{prop:injective}, we have
	\begin{equation}
		\fsfn{\nu}{\arint{\lambda}{q}}{\qq} \ge \fsfn{\nu'}{\arint{\lambda}{q}}{\qq} \qquad \implies \qquad
		\fsfn{\nu}{\arel{\lambda}{q}}{\qq} \le \fsfn{\nu'}{\arel{\lambda}{q}}{\qq},
	\end{equation}
	for all $\nu \le \nu'$.  Thus, the string functions of $\arel{\lambda}{q}$ are bounded above and below by $\tfsfn{\infty}{\arel{\lambda}{q}}{\qq}$ and $\tfsfn{-\infty}{\arel{\lambda}{q}}{\qq}$, respectively.  \cref{thm:limstrfnsexist} shows that the positive limits exist and we shall shortly see that the negative ones do too.

	Suppose first that $\sqrt{1+2q} \in \ZZ$ and let $\mu$ be the maximal solution in $\lambda$ of $\bilin{\mu}{\mu+2\wvec} = q$.  Then, $\airr{-\mu-\sroot}^+$ is a submodule of $\arel{\lambda}{q}$, by \cref{lem:simplesubmod}.  Thus, we have
	\begin{equation} \label{eq:secretequalities}
		\fsfn{-\infty}{\airr{-\mu-\sroot}^+}{\qq} \le \fsfn{-\infty}{\arel{\lambda}{q}}{\qq} \le \fsfn{\infty}{\arel{\lambda}{q}}{\qq} = \fsfn{\infty}{\airr{\mu+\sroot}^-}{\qq},
	\end{equation}
	where the last equality is \cref{lem:nsimplimstrfn}.  However, $\tsfn{-\infty}{\airr{-\mu-\sroot}^+} = \tsfn{\infty}{\conj \airr{-\mu-\sroot}^+} = \tsfn{\infty}{\airr{\mu+\sroot}^-}$, so the inequalities in \eqref{eq:secretequalities} are actually equalities.  It follows that $\arel{\lambda}{q}$ is stringy with the required string functions.

	It remains to consider the case when $\sqrt{1+2q} \notin \ZZ$ and so $\mu$ is the unique solution in $\lambda$ of $\bilin{\mu}{\mu+2\wvec} = q$.  Now, \cref{lem:simplesubmod} gives
	\begin{equation}
		\fsfn{-\infty}{\airr{\mu}^+}{\qq} \le \fsfn{-\infty}{\arel{\lambda}{q}}{\qq} \le \fsfn{\infty}{\arel{\lambda}{q}}{\qq} = \fsfn{\infty}{\airr{\mu+\sroot}^-}{\qq}
	\end{equation}
	in place of \eqref{eq:secretequalities}.  However, conjugating and applying \eqref{eq:nice} immediately gives
	\begin{equation} \label{eq:usefulsfnid}
		\fsfn{-\infty}{\airr{\mu}^+}{\qq} = \fsfn{\infty}{\airr{-\mu}^-}{\qq} = \fsfn{\infty}{\airr{\mu+\sroot}^-}{\qq}.
	\end{equation}
	The stringiness is therefore established as before, as is the identification of the string functions.
\end{proof}

For later use, we provide a strengthening of \cref{lem:simplesubmod} in the case where $\sqrt{1+2q} \notin \ZZ$.
\begin{proposition} \label{prop:2cfs}
	Choose $q \in \CC$ so that $\sqrt{1+2q} \notin \ZZ$.  Then, for each $\lambda \in \nsimpset{q}$, we have an short exact sequence
	\begin{equation} \label{eq:es}
		\dses{\airr{\mu}^+}{\arel{\lambda}{q}}{\airr{\mu+\sroot}^-},
	\end{equation}
	where $\mu$ denotes the (unique) solution of $\bilin{\mu}{\mu + 2 \wvec} = q$ in $\lambda$.
\end{proposition}
\begin{proof}
	By the proof of \cref{lem:simplesubmod}, we have $\airr{\mu}^+ \ira \arel{\lambda}{q}$ and $\airr{\mu}^+ \cong \aver{\mu}^+ \big/ \brac[\big]{\aver{\mu}^+ \cap \arint{\lambda}{q}}$.  It follows that
	\begin{equation}
		\frac{\arel{\lambda}{q}}{\airr{\mu}^+} \cong
		\left. \arel{\lambda}{q} \middle/ \frac{\aver{\mu}^+}{\aver{\mu}^+ \cap \arint{\lambda}{q}} \right. \cong
		\left. \frac{\arver{\lambda}{q}}{\arint{\lambda}{q}} \middle/ \frac{\aver{\mu}^+ + \arint{\lambda}{q}}{\arint{\lambda}{q}} \right. \cong
		\frac{\arver{\lambda}{q}}{\aver{\mu}^+ + \arint{\lambda}{q}}.
	\end{equation}
	Since $\airr{\mu+\sroot}^-$ is the unique simple quotient of $\arel{\lambda}{q}$ and $\arver{\lambda}{q}$, by \cref{lem:simpquot}, the \lcnamecref{prop:2cfs} will follow if we can show that $\aver{\mu}^+ + \arint{\lambda}{q} = \armax{\lambda}{q}$ in $\arver{\lambda}{q}$.

	The inclusion $\aver{\mu}^+ + \arint{\lambda}{q} \subseteq \armax{\lambda}{q}$ is clear, so suppose that $v \in \armax{\lambda}{q}$.  Without loss of generality, we may assume that $v$ is a weight vector.  Then, there exists $m$ such that $e_0^m v \in \arint{\lambda}{q}$, because the weight spaces of $\arint{\lambda}{q}$ and $\armax{\lambda}{q}$ coincide for sufficiently large $\sltwo$-weights, by \cref{lem:I=J}.  Moreover, there exists $n$ such that $f_0^n v \in \aver{\mu}^+$, by the \pbw{} theorem.  It follows that the image of $v$ in $\armax{\lambda}{q} \big/ \brac[\big]{\aver{\mu}^+ + \arint{\lambda}{q}}$ generates a finite-dimensional $\sltwo$-module.  As $\sqrt{1+2q} \notin \ZZ$, we have $\mu \notin \wlat$ by \eqref{eq:sols}, so this is impossible unless the image is $0$.  It follows that $\armax{\lambda}{q} = \aver{\mu}^+ + \arint{\lambda}{q}$ as required.
\end{proof}

We conclude with a cautionary example illustrating that our intuition with respect to composition factors of relaxed \hwms{} may need refining when $\sqrt{1+2q} \in \ZZ$.
\begin{example}
	Consider the $\asltwo$-module $\arver{-\wvec}{-1/2}$ at level $\kk = -1$.  Note that $\sqrt{1+2q} = 0$ and $\mu = -\wvec$.  The $\sltwo$-module of ground states therefore has exact sequence
	\begin{equation}
		\dses{\fver{-\wvec}^+}{\fden{-\wvec}{-1/2}}{\fver{\wvec}^-},
	\end{equation}
	in which both Verma modules are simple.  However, the corresponding short sequence
	\begin{equation} \label{eq:notes}
		\dses{\airr{-\wvec}^+}{\arel{-\wvec}{-1/2}}{\airr{\wvec}^-}
	\end{equation}
	of $\asltwo$-modules is \emph{not} exact.  The easiest way to see this is to compute the dimensions of the following weight spaces using the Shapovalov form on $\aver{-\wvec}^+$:
	\begin{equation}
		\airr{-\wvec}^+(3\wvec,\Delta+1), \qquad
		\airr{-\wvec}^+(\wvec,\Delta+1), \qquad
		\airr{-\wvec}^+(-\wvec,\Delta+1), \qquad
		\airr{-\wvec}^+(-3\wvec,\Delta+1).
	\end{equation}
	Here, we recall that $\Delta$ is the conformal weight of the ground states of $\arver{-\wvec}{-1/2}$.  These dimensions are $0$ (obviously), $0$ (because $e_{-1} v_{-\wvec}$ is singular in $\aver{-\wvec}^+$), $1$ and $2$, respectively.  Now, if \eqref{eq:notes} were exact, then we would have
	\begin{equation}
		\begin{aligned}
			\dim \arel{-\wvec}{-1/2}(-3\wvec,\Delta+1) &=
			\dim \airr{-\wvec}^+(-3\wvec,\Delta+1) + \dim \airr{\wvec}^-(-3\wvec,\Delta+1) \\ &=
			\dim \airr{-\wvec}^+(-3\wvec,\Delta+1) + \dim \airr{\wvec}^+(3\wvec,\Delta+1) = 2, \\ \text{and} \qquad
			\dim \arel{-\wvec}{-1/2}(-\wvec,\Delta+1) &=
			\dim \airr{-\wvec}^+(-\wvec,\Delta+1) + \dim \airr{\wvec}^-(-\wvec,\Delta+1) \\ &=
			\dim \airr{-\wvec}^+(-\wvec,\Delta+1) + \dim \airr{\wvec}^+(\wvec,\Delta+1) = 1.
		\end{aligned}
	\end{equation}
	However, this is impossible because $\arel{-\wvec}{-1/2}$ is stringy, by \cref{thm:nsimpstrfns}.

	We can isolate an additional composition factor of $\arel{-\wvec}{-1/2}$, beyond $\airr{-\wvec}^+$ and $\airr{\wvec}^-$, as follows.  First, prove that the following relations hold in $\arel{-\wvec}{-1/2}$ (the \lhss{} are annihilated by all positive modes):
	\begin{equation} \label{eq:rsvs}
		e_{-1} v_{\nu-\sroot} + \bilin{\nu - \wvec}{\wvec} h_{-1} v_{\nu} - \frac{1}{2} \norm{\nu - \wvec}^2 f_{-1} v_{\nu+\sroot} = 0, \qquad \text{for all $\nu \in -\wvec$.}
	\end{equation}
	Second, note that $f_{-1} v_{\wvec}$ is non-zero in $\arel{-\wvec}{-1/2}$ as the module it generates contains $v_{-\wvec}$.  Third, use \eqref{eq:rsvs} to show that $e_0 f_{-1} v_{\wvec}$ is a \hwv{} in $\arel{-\wvec}{-1/2} \big/ \airr{-\wvec}^+$.  We conclude that $\airr{\wvec}$ is also a composition factor of $\arel{-\wvec}{-1/2}$.  Note however that this analysis does not rule out the existence of further composition factors.  We illustrate the structure of $\arel{-\wvec}{-1/2}$ in \cref{fig:Ewith3CFs}.
\end{example}

\begin{figure}
	\begin{tikzpicture}[scale=1.3,>=latex]
		\foreach \x in {-3,-1,1,3}
			\foreach \y in {-2,-1,0}
				\node[wt] at (\x,\y) {};
		\node at (-1,0.3) {$v_{-\wvec}$};
		\node at (1,0.3) {$v_{\wvec}$};
		\draw[dotted] (-4.3,0) -- (-3.8,0);
		\draw[gray] (-3.7,0) -- (-1,0) -- (1.5,-2.5);
		\draw[dotted] (1.6,-2.6) -- (1.9,-2.9);
		\draw[dotted] (4.3,0) -- (3.8,0);
		\draw[gray] (3.7,0) -- (1,0) -- (-1.5,-2.5);
		\draw[dotted] (-1.6,-2.6) -- (-1.9,-2.9);
		\draw[dotted] (-4.3,-2.65) -- (-3.8,-2.4);
		\draw[gray] (-3.7,-2.35) -- (-1,-1) -- (1,-1) -- (3.7,-2.35);
		\draw[dotted] (4.3,-2.65) -- (3.8,-2.4);
		\node[circle,draw,fill=gray!25] at (-4,-1) {$\airr{-\wvec}^+$};
		\node[circle,draw,fill=gray!25] at (4,-1) {$\airr{+\wvec}^-$};
		\node[circle,draw,fill=gray!25] at (0,-2.5) {$\airr{\wvec}$};
		\draw[thick,->] (0.9,0) -- node[above] {$f_0$} (-0.9,0);
		\draw[thick,->] (0.9,-0.05) -- node[above,pos=0.9] {$f_{-1}$} (-0.9,-0.95);
		\draw[thick,->] (0.9,-0.95) -- node[above,pos=0.1] {$e_1$} (-0.9,-0.05);
	\end{tikzpicture}
\caption{A depiction of the structure of $\arel{-\wvec}{-1/2}$, showing three composition factors (though there could be more) and arrows indicating the $\asltwo$-action.  Black dots denote weights and are labelled by ground states when appropriate.  $\sltwo$-weights increase from left to right, while conformal weights increase from top to bottom.} \label{fig:Ewith3CFs}
\end{figure}
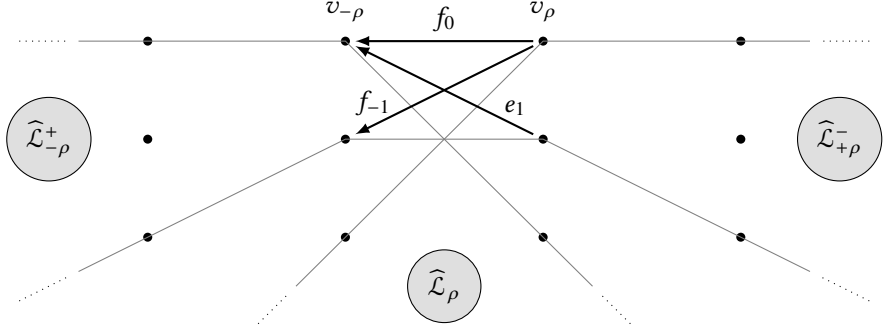

\section{Application to admissible-level $\slvoa{\kk}$-modules} \label{sec:voasl2}

We now apply the results of the previous \lcnamecref{sec:charsl2} to study the $\arel{\lambda}{q}$ that define modules over the simple affine \voa{} $\slvoa{\kk}$, where $\kk$ is an admissible level.  This means that $\kk$ has the form
\begin{equation}
	\kk + \dcox = \frac{u}{v}, \qquad u \in \ZZ_{\ge 2},\ v \in \ZZ_{\ge 1},\ \gcd \set{u,v} = 1,
\end{equation}
where we recall that the dual Coxeter number of $\sltwo$ is $\dcox = 2$.  As the conformal weights of any module over an affine \voa{} are fixed by the Sugawara construction, we shall set those of the ground states of $\arel{\lambda}{q}$ to be
\begin{equation} \label{eq:confwtsl2}
	\Delta = \Delta_q = \frac{q}{2 (\kk + \dcox)}.
\end{equation}

The $\arel{\lambda}{q}$ that define $\slvoa{\kk}$-modules are those with \cite{AdaVer95,RidRel15}
\begin{equation}
	q = q_{r,s} = \frac{1}{2} \brac*{\brac[\big]{r - \frac{u}{v} s}^2 - 1} = \frac{(vr-us)^2 - v^2}{2v^2}, \qquad r=1,\dots,u-1, \quad s=1,\dots,v-1.
\end{equation}
Note the ``Kac table''-type symmetry $q_{u-r,v-s} = q_{r,s}$ indicating coincidences amongst these relaxed \hwms{}.  Moreover, $u$ and $v$ being coprime gives $\sqrt{1+2q_{r,s}} = \abs[\big]{r - \frac{u}{v} s} \notin \ZZ$ which implies that we have $\abs*{\nsimpset{q_{r,s}}} = 2$.  In other words, there are two distinct cosets $\lambda \in \h^* / \rlat$, for each $r$ and $s$ (modulo the Kac symmetry), defining non-simple relaxed \hwms{} of the form $\arel{\lambda}{q_{r,s}}$.  Indeed, the $\mu \in \h^*$ satisfying $\bilin{\mu}{\mu+2\wvec} = q_{r,s}$ are given by
\begin{equation}
	\mu = \mu_{r,s} = \brac*{r-1 - \frac{u}{v} s} \fwt \qquad \text{and} \qquad \mu = \mu_{u-r,v-s} = \brac*{-r-1 + \frac{u}{v} s} \fwt,
\end{equation}
where we recall that $\fwt$ denotes the fundamental weight of $\sltwo$.

The structures of the non-simple relaxed \hwms{} $\arel{\lambda}{q_{r,s}}$, with $\sqrt{1+2q} \notin \ZZ$, are now immediate consequences of \cref{prop:2cfs}.  These structures were previously stated, without proof, in \cite{RidFus10} (see Eqs.~(4.14) and (4.29)), \cite{CreMod12} (see Eq.~(3.14) and the structure diagrams of Sec.~5.1) and \cite{CreMod13} (see Eq.~(4.3)).
\begin{theorem} \label{thm:esvoa}
	Each admissible-level $\slvoa{\kk}$-module $\arel{\mu_{r,s}}{q_{r,s}}$, where $r=1,\dots,u-1$ and $s=1,\dots,v-1$, is a non-split extension of the (conjugate) simple \hwm{} $\airr{\mu_{r,s}+\sroot}^- = \airr{-\mu_{u-r,v-s}}^-$ by the simple \hwm{} $\airr{\mu_{r,s}}^+$.  In other words, the following sequence is exact:
	\begin{equation} \label{eq:esvoa}
		\dses{\airr{\mu_{r,s}}^+}{\arel{\mu_{r,s}}{q_{r,s}}}{\airr{-\mu_{u-r,v-s}}^-}.
	\end{equation}
\end{theorem}

\begin{remark}
	Recall that the (non-simple) relaxed \hwm{} $\arel{\mu_{r,s}}{q_{r,s}}$ was chosen so that $f_0$ acts injectively.  Its conjugate therefore has an injective action of $e_0$ and is a non-split extension of $\airr{\mu_{u-r,v-s}}^+$ by $\airr{-\mu_{r,s}}^-$.  In particular, $\conj \arel{\mu_{r,s}}{q_{r,s}}$ is not isomorphic to $\arel{\mu_{u-r,v-s}}{q_{u-r,v-s}}$.
\end{remark}

Finally, we turn to the characters of the $\slvoa{\kk}$-modules $\arel{\lambda}{q_{r,s}}$, $r=1,\dots,u-1$ and $s=1,\dots,v-1$.  \cref{thm:simpstrfns,thm:nsimpstrfns} allow us to compute their string functions in terms of the limiting string functions of the $\airr{\mu_{r,s}+\sroot}^-$.  Indeed, \cref{eq:simlim=lim,eq:usefulsfnid} give
\begin{equation}
	\fsfn{\xi}{\arel{\lambda}{q_{r,s}}}{\qq} = \fsfn{\infty}{\airr{\mu_{r,s}+\sroot}^-}{\qq} = \fsfn{-\infty}{\airr{\mu_{r,s}}^+}{\qq}, \qquad \text{for all $\xi \in \lambda$,}
\end{equation}
independent of $\lambda \in \h^* / \rlat$.  The rightmost limiting string function can now be computed from the Kac-Wakimoto character formula \cite{KacMod88} because $\airr{\mu_{r,s}}^+$ is an admissible level-$\kk$ \hw{} $\asltwo$-module.  We write the character in the form \cite{IohFus01a}
\begin{equation} \label{eq:BGG}
	\fch{\airr{\mu_{r,s}}^+}{\zz;\qq} = \sum_{n \in \ZZ} \brac*{\fch{\aver{\mu_{2nu+r,s}}^+}{\zz;\qq} - \fch{\aver{\mu_{2nu-r,s}}^+}{\zz;\qq}},
\end{equation}
where the Verma module characters are given in \eqref{eq:chvermasl2}.  It is convenient at this point to reinstate the convention that characters and string functions are normalised by the factor $\qq^{-\cc/24}$, where
\begin{equation}
	\cc = \frac{3\kk}{\kk+\dcox} = 3 - \frac{6v^2}{uv}
\end{equation}
is the central charge of $\slvoa{\kk}$.

Now we can use the computation of the limiting string function for Verma modules in \cref{prop:vermasfnsl2}:
\begin{align} \label{eq:relsfn}
	\fsfn{-\infty}{\airr{\mu_{r,s}}^+}{\qq} &= \sum_{n \in \ZZ} \brac*{\fsfn{-\infty}{\aver{\mu_{2nu+r,s}}^+}{\qq} - \fsfn{-\infty}{\aver{\mu_{2nu-r,s}}^+}{\qq}} \\
	&= \frac{1}{\eta(\qq)^3} \sum_{n \in \ZZ} \brac*{\qq^{\Delta_{2nu+r,s} - \cc/24 + 1/8} - \qq^{\Delta_{2nu-r,s} - \cc/24 + 1/8}} \notag \\
	&= \frac{1}{\eta(\qq)^3} \sum_{n \in \ZZ} \brac*{\qq^{\Delta^{\textup{Vir}}_{2nu+r,s} - \cc^{\textup{Vir}}/24 + 1/24} - \qq^{\Delta^{\textup{Vir}}_{2nu-r,s} - \cc^{\textup{Vir}}/24 + 1/24}}. \notag
\end{align}
Here, $\Delta_{r,s} = \Delta_{q_{r,s}}$ and the Virasoro conformal weights and central charge are given by the usual formulae:
\begin{equation}
	\Delta^{\textup{Vir}}_{r,s} = \frac{(vr-us)^2 - (v-u)^2}{4uv}, \qquad \cc^{\textup{Vir}} = 1 - \frac{6(v-u)^2}{uv}.
\end{equation}
Recognising in \eqref{eq:relsfn} the character
\begin{equation} \label{eq:virch}
	\chi^{\textup{Vir}}_{r,s}(\qq) = \frac{\qq^{(1-\cc^{\textup{Vir}})/24}}{\eta(\qq)} \sum_{n \in \ZZ} \brac*{\qq^{\Delta^{\textup{Vir}}_{2nu+r,s}} - \qq^{\Delta^{\textup{Vir}}_{2nu-r,s}}}
\end{equation}
of the simple \hw{} Virasoro module of conformal weight $\Delta^{\textup{Vir}}_{r,s}$ and central charge $\cc^{\textup{Vir}}$, \cref{eq:simlim=lim} gives the string functions, and thence the characters, of all the $\arel{\lambda}{q_{r,s}}$.  This proves a character formula for these modules that was originally conjectured in \cite{CreMod13}.

\begin{theorem} \label{thm:chsl2}
	The characters of the admissible-level $\slvoa{\kk}$-modules $\arel{\lambda}{q_{r,s}}$, with $\lambda \in \h^* / \rlat$, $r=1,\dots,u-1$ and $s=1,\dots,v-1$, are given by
	\begin{equation}
		\fch{\arel{\lambda}{q_{r,s}}}{\zz;\qq} = \frac{\chi^{\textup{Vir}}_{r,s}(\qq)}{\eta(\qq)^2} \sum_{\mu \in \lambda} \zz^{\mu}.
	\end{equation}
\end{theorem}

\section{Relaxed \hw{} $\aosp$-modules} \label{sec:relaxedosp12}

We now generalise our study of relaxed \hwms{} over $\asltwo$ to $\aosp$.  We follow a similar strategy as before, but content ourselves with only describing those parts of the arguments that are not just straightforward generalisations of their $\asltwo$ analogues.  The main differences arise because the intended application to modules of admissible-level \svoas{} $\ospvoa{\kk}$ requires us to analyse both the untwisted (\ns{}) and twisted (Ramond) sectors.

\subsection{Simple weight $\osp$-modules} \label{sec:swosp}

The simple basic classical Lie superalgebra $\osp$ has basis $\set{e,x,h,y,f}$, where $e$, $h$ and $f$ are even while $x$ and $y$ are odd.  As the notation suggests, the even subalgebra of $\osp$ is isomorphic to $\sltwo$ and so the commutation rules \eqref{eq:commsl2} continue to hold.  The remaining (anti)commutation relations involving the basis elements may be taken to be
\begin{equation} \label{eqn:commosp}
	\begin{aligned}
		\comm{e}{x}&=0, &\comm{h}{x}&=x, &\comm{f}{x}&=-y, \\
		\comm{e}{y}&=-x, &\comm{h}{y}&=-y, &\comm{f}{y}&=0, \\
		\acomm{x}{x}&=2e, &\acomm{x}{y}&=h, &\acomm{y}{y}&=-2f.
	\end{aligned}
\end{equation}
The non-zero entries of the (rescaled) Killing form, in this basis, are
\begin{equation}
	\killing{h}{h} = 2, \quad \killing{e}{f} = \killing{f}{e} = 1, \quad \killing{x}{y} = -\killing{y}{x} = 1.
\end{equation}
The Cartan subalgebra is chosen to be $\h = \CC h$ and the quadratic Casimir to be
\begin{equation}
	\cas' = \frac{1}{2} h^2 + ef + fe - \frac{1}{2} xy + \frac{1}{2} yx.
\end{equation}
In $\envalg{\osp}$, there is also the super-Casimir \cite{ArnCas97} given by
\begin{equation}
	\scas = xy - yx + \frac{1}{2}.
\end{equation}
It is not central, but rather commutes with $e$, $h$ and $f$, while it anticommutes with $x$ and $y$.  Note that $\scas^2 = 2\cas' + \frac{1}{4}$.

Let $\fwt \in \h^*$ and $\sroot = 2 \fwt$ denote the fundamental weight and highest root of $\osp$.  The (odd) simple root is then $\frac{1}{2} \sroot = \fwt$ and the Weyl vector is $\wvec = \frac{1}{2} \fwt$.  Let $\wlat = \rlat = \ZZ \fwt$ denote the weight and root lattices, while $\rlat^0 = \ZZ \sroot$ denotes the even root lattice.  $\wlat_{\ge} = \ZZ_{\ge 0} \fwt$ again denotes the dominant integral weights.  We induce the Killing form to a bilinear form $\bilin{\cdot}{\cdot}$ on $\h^*$, noting that the rescaling again normalises the latter so that $\norm{\sroot}^2 = 2$.

The classification of simple weight $\osp$-modules follows a similar pattern to that of $\sltwo$-modules (\cref{prop:sl2simples}).  We recall our assumption (\cref{sec:relaxed}) that weight $\osp$-modules are $\ZZ_2$-graded by parity.  This means that an $\osp$-module decomposes into the direct sum of an even and an odd subspace, which are both preserved by the even elements $e$, $h$ and $f$ but are swapped by the odd elements $x$ and $y$.  There is an obvious \emph{parity-reversal} functor $\parrev$ on any category of $\ZZ_2$-graded $\osp$-modules given by exchanging the even and odd subspaces.  We note that $\scas$-eigenvalues are constant on the even and odd subspaces of a simple $\osp$-module, taking values $\sigma$ and $-\sigma$, respectively, for some $\sigma \in \CC$.

As we did for $\sltwo$, it is convenient to introduce a family of subsets, this time parametrised by $\sigma \in \CC$:
\begin{equation} \label{eq:defLambdao}
	\nsimpseto{\sigma} = \set*{[\lambda] \in \h^* / \rlat^0 \st \norm{\mu}^2 = \tfrac{1}{2} \brac*{\sigma - \tfrac{1}{2}}^2\ \text{for some}\ \mu \in [\lambda]}.
\end{equation}
This facilitates the following classification of simple weight $\osp$-modules.
\begin{proposition}[\protect{\cite[Thm.~2]{RidAdm17}}] \label{prop:ospsimples}
	Every simple ($\ZZ_2$-graded) weight $\osp$-module (with finite-dimensional weight spaces) is either isomorphic to a member of one of the following families or its parity-reversal is.
	\begin{enumerate}
		\item The finite-dimensional modules $\ffino{\mu}$ with even highest weight $\mu \in \wlat_{\ge}$ and lowest weight $-\mu$.
		\item The \hw{} Verma modules $\fvero{\mu}^+$ with even highest weight $\mu \notin \wlat_{\ge}$.
		\item The \lw{} Verma modules $\fvero{\mu}^-$ with even lowest weight $\mu \notin -\wlat_{\ge}$.
		\item The dense modules $\fdeno{\lambda}{\sigma}$ whose even weight vectors have weights in $\lambda \in \h^* / \rlat^0$ and $\scas$-eigenvalue $\sigma \in \CC$, where $\lambda \notin \nsimpseto{\sigma}$.
	\end{enumerate}
	All of these modules have one-dimensional weight spaces.
\end{proposition}
\noindent Note that the weight support of the dense module $\fdeno{\lambda}{\sigma}$ is actually $\lambda \cup (\lambda + \fwt)$, the second $\rlat^0$-coset corresponding to odd vectors.  We parametrise these modules by their even weight supports because of the obvious isomorphisms $\fdeno{\lambda+\fwt}{\sigma} \cong \parrev \fdeno{\lambda}{-\sigma}$.  As for $\sltwo$ (see \cref{sec:swsl2}), the Weyl reflection of $\osp$ defines a functor on $\osp$-modules that exchanges $\fvero{\mu}^+$ with $\fvero{-\mu}^-$.

\subsection{Non-simple dense $\osp$-modules} \label{sec:denseosp}

The action of $e$, $x$, $y$ and $f$ on the simple dense modules $\fdeno{\lambda}{\sigma}$, $\lambda \notin \nsimpseto{\sigma}$, is injective.  We shall therefore choose basis vectors $v_{\mu}$, $\mu \in \lambda$, so that
\begin{subequations} \label{eq:ospaction}
	\begin{equation}
		e v_{\mu} = \beta_{\mu}^2 v_{\mu+\sroot}, \quad
		x v_{\mu} = \beta_{\mu} v_{\mu+\fwt}, \quad
		h v_{\mu} = \bilin{\mu}{\sroot} v_{\mu}, \quad
		y v_{\mu} = v_{\mu-\fwt}, \quad
		f v_{\mu} = -v_{\mu-\sroot},
	\end{equation}
	where
	\begin{equation}
		\beta_{\mu} = \frac{1}{2} \sqbrac*{\bilin{\mu}{\sroot} - (-1)^{\parity{v_{\mu}}} \sigma + \frac{1}{2}}
	\end{equation}
\end{subequations}
and $\parity{v_{\mu}} \in \set{0,1}$ denotes the parity of $v_{\mu}$.  This action is observed to be polynomial in $\mu \in \h^*$, up to the parity-dependent sign.  This is not a major obstacle to the analysis to follow because we can just restrict to the even and odd subspaces when we wish to exploit this polynomial dependence.  Our first task is to define non-simple indecomposables $\fdeno{\lambda}{\sigma}$, with $\lambda \in \nsimpseto{\sigma}$, to complete this family.  We shall construct them so that $y$ and $f$ continue to act injectively, hence $v_{\mu}$ may be chosen such that \eqref{eq:ospaction} continues to hold.

We therefore fix $\sigma \in \CC$ and define indecomposable dense $\osp$-modules $\fdeno{\lambda}{\sigma}$, with $\lambda \in \nsimpseto{\sigma}$, by inducing from the centraliser $\CC[h,\scas]$ of $\h$ in $\envalg{\osp}$ as follows.  Let $v$ be an \emph{even} eigenvector of $h$ and $\scas$, so that $hv = \lambda(h) v$ and $\scas v = \sigma v$ for some $\lambda \in \h^*$ and $\sigma \in \CC$ satisfying $[\lambda] \in \nsimpseto{\sigma}$.  The structure of the $\osp$-module induced from the $\CC[h,\scas]$-module $\CC v$ then depends on the relative ordering (by real parts of Dynkin labels) between $\lambda$ and the solutions
\begin{equation} \label{eq:ospsols}
	\mu = \pm \brac*{\sigma - \tfrac{1}{2}} \fwt
\end{equation}
of $\norm{\mu}^2 = \frac{1}{2} \brac*{\sigma - \frac{1}{2}}^2$.  We take $\fdeno{\lambda}{\sigma} = \fdeno{[\lambda]}{\sigma}$ to be the induced module with $\lambda$ larger than all solutions, so that $y$ and $f$ act injectively.  It follows that $\fdeno{\lambda}{\sigma}$ may have \hwvs{}, but no \lwvs{}.  Indeed, $v_{\mu}$ will be an even \hwv{} if $\mu = \brac*{\sigma - \tfrac{1}{2}} \fwt \in \lambda$ and $yv_{\mu}$ will be an odd \hwv{} if $\mu = -\brac*{\sigma - \tfrac{1}{2}} \fwt \in \lambda$.

\begin{example}
	Consider the case $\sigma = \frac{1}{2}$, so that \eqref{eq:ospsols} has the unique solution $\mu = 0$.  Then, $[\lambda] = [0] = \rlat^0$ and $\fdeno{0}{1/2}$ has two \hwvs{}: $v_0$ and $yv_0 = v_{-\fwt}$.  Its (unique) composition series is therefore
\begin{equation}
	0 \subset \parrev \fvero{-\fwt}^+ \subset \fvero{0}^+ \subset \fdeno{[0]}{1/2},
\end{equation}
with composition factors $\parrev \fvero{-\fwt}^+$, $\ffino{0}$ and $\parrev \fvero{\fwt}^-$.
\end{example}

This case generalises: both solutions \eqref{eq:ospsols} belong to the same $\rlat^0$-coset if and only if $\sigma \in \ZZ + \frac{1}{2}$.  In this case, take $\mu \in \wlat_{\ge}$ to be the maximal solution in $\lambda = [\mu] \in \h^* / \rlat^0$.  The composition series of $\fden{\lambda}{\sigma}$ thus depends on the sign of $\sigma$.  Specifically, if $\sigma \in \ZZ + \frac{1}{2}$ and $\sigma > 0$, then the series is
\begin{equation}
	0 \subset \parrev \fvero{-\mu-\fwt}^+ \subset \fvero{\mu}^+ \subset \fdeno{\lambda}{\sigma}
\end{equation}
and the composition factors are $\parrev \fvero{-\mu-\fwt}^+$, $\ffino{\mu}$ and $\parrev \fvero{\mu+\fwt}^-$.  However, for $\sigma \in \ZZ + \frac{1}{2}$ and $\sigma < 0$, the series is instead
\begin{equation}
	0 \subset \fvero{-\mu}^+ \subset \parrev \fvero{\mu-\fwt}^+ \subset \fdeno{\lambda}{\sigma},
\end{equation}
with composition factors $\fvero{-\mu}^+$, $\parrev \ffino{\mu+\fwt}$ and $\fvero{\mu}^-$.

The remaining case corresponds to the two solutions \eqref{eq:ospsols} belonging to different $\rlat^0$-cosets, whence $\sigma \notin \ZZ + \frac{1}{2}$.  This leads to two inequivalent indecomposable dense $\osp$-modules $\fdeno{\lambda_{\pm}}{\sigma}$, where $\lambda_{\pm} = [\mu_{\pm}] = [\pm(\sigma - \tfrac{1}{2}) \fwt]$.  The composition series are
\begin{equation}
	0 \subset \fvero{\mu_+}^+ \subset \fdeno{\lambda_+}{\sigma} \qquad \text{and} \qquad
	0 \subset \parrev \fvero{\mu_--\fwt}^+ \subset \fdeno{\lambda_-}{\sigma},
\end{equation}
with respective composition factors $\fvero{\mu_+}^+$, $\parrev \fvero{\mu_++\fwt}^-$ and $\parrev \fvero{\mu_--\fwt}^+$, $\fvero{\mu_-}^-$.

\subsection{Relaxed \hw{} $\aosp$-modules} \label{sec:rhwosp}

As in \cref{sec:rhwsl2}, we may induce an indecomposable weight $\osp$-module to a relaxed Verma $\aosp$-module in category $\categ{R}$, once we choose eigenvalues $\kk$ and $\Delta$ for $K$ and $L_0$.  These induced modules are $\aosp$-modules, but are often said to belong to the \emph{\ns{}} sector for historical reasons.  Such \ns{} modules will be denoted by adding hats and $\NS$ symbols to the $\osp$-modules that they were induced from.  Thus, we have Verma modules $\averns{\mu}^{\pm}$, parabolic Vermas $\afinns{\mu}$ and relaxed Vermas $\arverns{\lambda}{\sigma}$.  Quotienting each by the maximal submodule whose intersection with the space of ground states is zero results in more \ns{} modules: $\airrns{\mu}^{\pm}$, $\airrns{\mu}$ and $\arelns{\lambda}{\sigma}$ (respectively).  All are simple except for the $\arelns{\lambda}{\sigma}$ with $\lambda \in \nsimpseto{\sigma}$.

In many applications, those of \cref{sec:voaosp} for instance, one also needs to consider a twisted version of $\aosp$ in which the indices of $x_n$ and $y_n$ are required to belong to $\ZZ + \frac{1}{2}$ instead of $\ZZ$.  We shall denote this twisted version by $\raosp$ and refer to its modules as the \emph{Ramond} sector of $\aosp$.  Because the zero modes of the Ramond sector omit $x_0$ and $y_0$, we construct relaxed Verma $\aosp$-modules in the Ramond sector by inducing indecomposable weight $\sltwo$-modules.  The notation for the result adds a hat to the $\sltwo$-module being induced, just as in the \ns{} sector, but adds an \R{} symbol instead.  Thus, the Ramond sector has Verma modules $\averr{\mu}^{\pm}$, parabolic Vermas $\afinr{\mu}$ and relaxed Vermas $\arverr{\lambda}{q}$.  As above, quotienting each by its maximal submodule whose intersection with the space of ground states is zero gives new Ramond modules: $\airrr{\mu}^{\pm}$, $\airrr{\mu}$ and $\arelr{\lambda}{q}$ (respectively).  Again, these modules are all simple except for the $\arelr{\lambda}{q}$ with $\lambda \in \nsimpset{q}$.

Finally, the Weyl-reflection functor of $\osp$ lifts to a \emph{conjugation} functor on $\aosp$-modules that we shall (again) denote by $\conj$.  We have, for example, $\conj \airrns{\mu}^+ \cong \airrns{-\mu}^-$ and $\conj \airrr{\mu}^+ \cong \airrr{-\mu}^-$.  We emphasise that an \R{} label indicates that the module was induced from an $\sltwo$-module (which may otherwise share notation with a similar $\osp$-module) so that its parametrisation must always be understood in the context of $\sltwo$ data.

\section{Relaxed $\aosp$-modules and their (super)characters} \label{sec:charosp}

We now turn to the string functions of the \ns{} and Ramond relaxed \hw{} $\aosp$-modules $\arelns{\lambda}{\sigma}$ and $\arelr{\lambda}{q}$.  Their computation will only be outlined here as many of the details and proofs follow in an almost identical fashion to those detailed for $\asltwo$ in \cref{sec:charsl2}.

\subsection{String functions} \label{sec:stringosp}

The character of a \ns{} or Ramond level-$\kk$ weight module $\aMod{M}$ over $\aosp$ is still given by \eqref{eq:defchar} and string functions are likewise defined by \eqref{eq:defsfn}.  We shall also consider the supercharacter of $\aMod{M}$ which is given, at least for indecomposable weight modules, by inserting $(-1)^{\parity{\mu}}$ into the sum in \eqref{eq:defchar}, where $\parity{\mu} \in \set{0,1}$ denotes the parity of the weight vectors in $\aMod{M}$ whose $\osp$-weight is $\mu \in \h^*$.

\begin{example}
	The character and non-zero string functions of the level-$\kk$ \ns{} relaxed Verma module $\arverns{\lambda}{\sigma}$, for $\lambda \in \h^* / \rlat^0$ and $\sigma \in \CC$, are
	\begin{equation}
		\begin{aligned}
			\fch{\arverns{\lambda}{\sigma}}{\zz;\qq} &= \qq^{\Delta + 1/24} \frac{\fjth{2}{1;\qq}}{2 \eta(\qq)^4} \sqbrac*{\sum_{\mu \in \lambda} \zz^{\mu} + \sum_{\mu \in \lambda+\fwt} \zz^{\mu}} \\
			\Ra \quad \fsfn{\mu}{\arverns{\lambda}{\sigma}}{\qq} &= \qq^{\Delta + 1/24} \frac{\fjth{2}{1;\qq}}{2 \eta(\qq)^4}, \quad \text{if $\mu \in \lambda \cup (\lambda + \fwt$).}
		\end{aligned}
	\end{equation}
	It follows that $\arverns{\lambda}{\sigma}$ is stringy.  The Ramond relaxed Verma character and non-zero string functions are, however, given by
	\begin{equation}
		\begin{aligned}
			\fch{\arverr{\lambda}{q}}{\zz;\qq} &= \frac{\qq^{\Delta + 1/6}}{2 \eta(\qq)^4} \sqbrac*{\sum_{\mu \in \lambda} \brac[\big]{\fjth{3}{1;\qq} + \fjth{4}{1;\qq}} \zz^{\mu} + \sum_{\mu \in \lambda+\fwt} \brac[\big]{\fjth{3}{1;\qq} - \fjth{4}{1;\qq}} \zz^{\mu}} \\
			\Ra \quad \fsfn{\mu}{\arverr{\lambda}{q}}{\qq} &=
			\begin{cases*}
				\dfrac{\qq^{\Delta + 1/6}}{2 \eta(\qq)^4} \brac[\big]{\fjth{3}{1;\qq} + \fjth{4}{1;\qq}}, & \text{if $\mu \in \lambda$,} \\
				\dfrac{\qq^{\Delta + 1/6}}{2 \eta(\qq)^4} \brac[\big]{\fjth{3}{1;\qq} - \fjth{4}{1;\qq}}, & \text{if $\mu \in \lambda + \fwt$.}
			\end{cases*}
		\end{aligned}
	\end{equation}
	$\arverr{\lambda}{q}$ is therefore not stringy.  Here, $\jth{j}$ denotes the Jacobi theta functions (with the conventions of \cite[App.~B]{RidSL208}).

	The supercharacters of these relaxed Verma modules may be obtained from the above character formulae by replacing, in each, the sum of the two sums by their difference.
\end{example}

The previous Ramond example inspires us to make an alternative definition.
\begin{definition}
	A level-$\kk$ Ramond weight module $\aMod{M}$ is said to be \emph{\R-stringy} if its non-zero string functions $\tsfn{\mu}{\aMod{M}}$ depend only on whether $\mu$ belongs to its even or odd weight support.
\end{definition}
\noindent Obviously, the $\arverr{\lambda}{q}$ are \R-stringy (as are the $\arverns{\lambda}{\sigma}$).  Given an indecomposable level-$\kk$ Ramond weight module $\aMod{M}$, so that the even and odd weight supports are the $\rlat^0$-cosets $[\mu]$ and $[\mu+\fwt]$, respectively, for some $\mu \in \h^*$, we thus have two distinct notions of limiting string function:
\begin{equation}
	\fevsfn{\pm\infty}{\aMod{M}}{\qq} = \lim_{m \to \pm\infty} \fsfn{\mu + 2m \fwt}{\aMod{M}}{\qq}, \qquad
	\fodsfn{\pm\infty}{\aMod{M}}{\qq} = \lim_{m \to \pm\infty} \fsfn{\mu + (2m+1) \fwt}{\aMod{M}}{\qq}.
\end{equation}
We call these the limiting even and odd string functions, respectively.

\subsection{Coherent families and Shapovalov forms} \label{sec:cohosp}

Recall that in \cref{sec:relaxedosp12}, we defined $\arelns{\lambda}{\sigma}$ ($\arelr{\lambda}{q}$) to be the quotient of the relaxed Verma module $\arverns{\lambda}{\sigma}$ ($\arverr{\lambda}{q}$) by the maximal submodule $\arintns{\lambda}{\sigma}$ ($\arintr{\lambda}{q}$) whose intersection with the space of ground states is zero.  There are thus four types of relaxed coherent families to consider:
\begin{equation}
	\avcohns{\sigma} = \bigoplus_{\lambda \in \h^* / \rlat^0} \arverns{\lambda}{\sigma}, \quad
	\acohns{\sigma} = \bigoplus_{\lambda \in \h^* / \rlat^0} \arelns{\lambda}{\sigma}, \quad
	\avcohr{q} = \bigoplus_{\lambda \in \h^* / \rlat} \arverr{\lambda}{q}, \quad
	\acohr{q} = \bigoplus_{\lambda \in \h^* / \rlat} \arelr{\lambda}{q}.
\end{equation}
We let $\armaxns{\lambda}{\sigma}$ and $\armaxr{\lambda}{q}$ denote the maximal proper submodules of $\arverns{\lambda}{\sigma}$ and $\arverr{\lambda}{q}$ , respectively.

The first task is to construct analogues of Shapovalov forms on $\avcohns{\sigma}$ and $\avcohr{q}$.  For this, we need an adjoint on $\envalg{\aosp}$ and $\envalg{\raosp}$.  A convenient choice is
\begin{equation} \label{eq:adjointosp}
	e_n^{\dag} = f_{-n}, \quad x_n^{\dag} = \ii y_{-n}, \quad h_n^{\dag} = h_{-n}, \quad y_n^{\dag} = -\ii x_{-n}, \quad f_n^{\dag} = e_{-n}, \quad K^{\dag} = K, \quad L_0^{\dag} = L_0.
\end{equation}
We emphasise that ${}^{\dag}$ is taken to be a linear antiautomorphism, not an antilinear one, because the Shapovalov forms are intended to be bilinear, not sesquilinear.  With this understood, it is easy to check that ${}^{\dag}$ is involutive.

The ground states of $\arverns{\lambda}{\sigma}$ form a coherent family over $\osp$, hence we may define the Shapovalov form $\inner{\cdot}{\cdot}_{\nu}$ by choosing, for each $\lambda \in \h^* / \rlat^0$, a ground state $v_{\nu}$, as in \eqref{eq:ospaction}, that generates $\arverns{\lambda}{\sigma}$:
\begin{equation}
	\inner{v_{\nu}}{v_{\nu}}_{\nu} = 1 \quad \text{and} \quad
	\inner{U v_{\nu}}{V v_{\nu}}_{\nu} = \inner{v_{\nu}}{U^{\dag} V v_{\nu}}_{\nu} = \left. \beta(U^{\dag} V) \right\rvert_{h \mapsto \nu(h), \Sigma \mapsto \sigma}, \qquad \text{for all $U,V \in \envalgk{\aosp}$.}
\end{equation}
Here, $\beta \colon \envalgk{\aosp} \to \CC[h,\scas]$ is the projection whose kernel is spanned by the \pbw{} monomials, ordered so that indices increase, that have a non-zero index or have non-zero $\osp$-weight.  In the Ramond case, the coherent family of ground states is instead over $\sltwo$ so the definition of the Shapovalov forms is as in \eqref{eq:defshap2} except that the \uea{} is that of $\raosp$.  A Shapovalov form on each affine coherent family $\avcohns{\sigma}$ or $\avcohr{q}$ is then obtained as a direct sum of forms $\inner{\cdot}{\cdot}_{\nu}$ over $[\nu] \in \h^* / \rlat^0$ or $[\nu] \in \h^* / \rlat$, respectively.

Consider now the weight space $\avcohns{\sigma}(\nu, \Delta^{\NS}_{\sigma}+n)$ of $\osp$-weight $\nu$ and conformal weight $\Delta^{\NS}_{\sigma}+n$, where $n \in \ZZ_{\ge 0}$.  A basis for this space consists of the $U v_{\nu+n\sroot}$ in which $U$ is a \pbw{} monomial of $\envalgk{\aosp^{\le}}$, ordered so that indices increase, with $\osp$-weight $-n\sroot$ and conformal grade $n$, such that the exponents of $e_0$, $x_0$ and $h_0$ are all $0$. Then, the analogue of \cref{lem:I=J} shows that the dimension of $\acohns{\sigma}(\nu, \Delta^{\NS}_{\sigma}+n)$ is equal, for sufficiently large $\nu$, to the rank of the matrix whose entries are the values $\inner{U v_{\nu+n\sroot}}{V v_{\nu+n\sroot}}_{\nu}$ of the Shapovalov form applied to the basis elements.  A similar construction identifies $\dim \acohr{q}(\nu, \Delta^{\R}_{q}+m)$, $m \in \frac{1}{2} \ZZ_{\ge 0}$, with the rank of the corresponding Shapovalov matrix, again for sufficiently large $\nu$.  The proof of \cref{lem:rankconst} now readily generalises and we arrive at the following result.
\begin{proposition} \label{thm:limstrfnsexistosp}
	For given $\sigma, q \in \CC$, the positive limiting string functions $\tsfn{\infty}{\arelns{\lambda}{\sigma}}$ and $\trsfn{\infty}{\pm}{\arelr{\lambda}{q}}$ all exist and are $\lambda$-independent.
\end{proposition}
\noindent We remark that the Ramond results follow by restricting to $m \in \ZZ_{\ge 0}$, for the even limiting string functions, and to $m \in \ZZ_{\ge 0} + \frac{1}{2}$, for the odd ones.

\subsection{Stringiness of relaxed modules} \label{sec:stringyosp}

The $\aosp$-analogues of \cref{lem:intersection,prop:injective,lem:inj=stringy} are clear.  We summarise them along with the analogue of \cref{thm:simplestringy} for convenience.
\begin{proposition} \label{prop:simplestringyosp}
	\leavevmode
	\begin{itemize}
		\item Both $y_0$ and $f_0$ act injectively on $\arverns{\lambda}{\sigma}$, while $f_0$ acts injectively on $\arverr{\lambda}{q}$.
		\item If $\lambda \notin \nsimpseto{\sigma}$, then $e_0$ and $x_0$ also act injectively on $\arverns{\lambda}{\sigma}$ and, thus, $\arelns{\lambda}{\sigma}$ is stringy.
		\item If $\lambda \notin \nsimpset{q}$, then $e_0$ also acts injectively on $\arverr{\lambda}{q}$ and, thus, $\arelr{\lambda}{q}$ is \R-stringy.
	\end{itemize}
\end{proposition}
To identify these string functions, we employ \cref{thm:limstrfnsexistosp} to identify them with the positive limiting string functions of the non-simple $\arelns{\lambda}{\sigma}$, $\lambda \in \nsimpseto{\sigma}$, and $\arelr{\lambda}{q}$, $\lambda \in \nsimpset{q}$.  The explicit constructions in \cref{sec:denseosp}, combined with the simple quotient analogues of \cref{lem:simpquot,lem:nsimplimstrfn}, now lead to the following conclusions in the \ns{} sector.  Their Ramond counterparts follow similarly by adapting the $\asltwo$ results to $\raosp$.
\begin{theorem} \label{thm:simpstrfnsosp}
	\leavevmode
	\begin{itemize}
		\item If $\sigma \notin \ZZ + \frac{1}{2}$, then the non-zero string functions of the simple relaxed \hwms{} $\arelns{\lambda}{\sigma}$, $\lambda \notin \nsimpseto{\sigma}$, are
		\begin{subequations} \label{eq:strfnsNS}
			\begin{align}
				\fsfn{\nu}{\arelns{\lambda}{\sigma}}{\qq} &= \fsfn{\infty}{\airrns{\mu+\fwt}^-}{\qq} = \fsfn{\infty}{\airrns{\mu}^-}{\qq}
				\qquad \text{for all}\ \nu \in \lambda,
			\intertext{where $\mu$ is either of $\mu_{\pm} = \pm \brac{\sigma - \frac{1}{2}} \fwt$.  If $\sigma \in \ZZ + \frac{1}{2}$, then the non-zero string functions are instead}
				\fsfn{\nu}{\arelns{\lambda}{\sigma}}{\qq} &=
				\begin{cases*}
					\fsfn{\infty}{\airrns{\mu_+ + \fwt}^-}{\qq}, & if $\sigma>0$, \\
					\fsfn{\infty}{\airrns{\mu_-}^-}{\qq}, & if $\sigma<0$.
				\end{cases*}
			\end{align}
		\end{subequations}
		\item Similarly, the non-zero even and odd string functions of the simple relaxed \hwms{} $\arelr{\lambda}{q}$, $\lambda \notin \nsimpset{q}$, are
		\begin{equation} \label{eq:strfnsR}
			\frsfn{\nu}{\pm}{\arelr{\lambda}{q}}{\qq} = \frsfn{\infty}{\pm}{\airrr{\mu+\sroot}^-}{\qq}, \qquad \text{for all}\ \nu \in \lambda,
		\end{equation}
		where $\mu$ now denotes any solution of $\bilin{\mu}{\mu+2\wvec} = q$, if $\sqrt{1+2q} \notin \ZZ$, and the maximal such solution, if $\sqrt{1+2q} \in \ZZ$.
	\end{itemize}
\end{theorem}

It only remains to demonstrate the stringiness of the non-simple $\arelns{\lambda}{\sigma}$ and $\arelr{\lambda}{q}$.  This is straightforward, but a little tedious because the \ns{} analogue of \cref{lem:simplesubmod} now identifies the simple submodule $\aMod{M} \ira \arelns{\lambda}{\sigma}$ in terms of four separate cases tabulated as follows.
\begin{center}
	\renewcommand{\arraystretch}{1.2}
	\begin{tabular}{C|CC}
		\aMod{M} & \mu = \brac*{\sigma - \tfrac{1}{2}} \fwt & \mu = -\brac*{\sigma - \tfrac{1}{2}} \fwt \\
		\hline
		\sigma \in \ZZ + \tfrac{1}{2} & \parrev \airr{-\mu-\fwt}^+ & \airr{-\mu}^+ \\
		\sigma \notin \ZZ + \tfrac{1}{2} & \airr{\mu}^+ & \parrev \airr{\mu+\fwt}^+
	\end{tabular}
\end{center}
Here, $\mu$ denotes the maximal solution of $\norm{\mu}^2 = \frac{1}{2} \brac*{\sigma - \frac{1}{2}}^2$ in $\lambda$.  The Ramond version has only two cases, just like $\asltwo$:  $\airrr{-\mu-\sroot{}}^+ \ira \arelr{\lambda}{q}$, if $\sqrt{1+2q} \in \ZZ$, and otherwise $\airrr{\mu}^+ \ira \arelr{\lambda}{q}$.  For this sector, $\mu$ is the maximal solution of $\bilin{\mu}{\mu+2\wvec} = q$.  Applying the proof methods of \cref{thm:nsimpstrfns} to these six cases, we arrive at the desired conclusion.
\begin{theorem} \label{thm:nsimpstrfnsosp}
	\leavevmode
	\begin{itemize}
		\item If $\lambda \in \nsimpseto{\sigma}$, then $\arelns{\lambda}{\sigma}$ is stringy and its non-zero string functions are given by \eqref{eq:strfnsNS}.
		\item Similarly, if $\lambda \in \nsimpset{q}$, then $\arelr{\lambda}{q}$ is \R-stringy and its non-zero string functions are given by \eqref{eq:strfnsR}.
	\end{itemize}
\end{theorem}

Finally, we present the analogue of \cref{prop:2cfs}.  Again, the $\osp$ proof is virtually identical to the $\sltwo$ one.
\begin{proposition} \label{prop:2cfsosp}
	\leavevmode
	\begin{itemize}
		\item Let $\sigma \notin \ZZ + \frac{1}{2}$.  Then for each $\lambda \in \nsimpseto{\sigma}$, there is a unique solution $\mu$ of $\norm{\mu}^2 = \frac{1}{2} \brac*{\sigma - \frac{1}{2}}^2$ in $\lambda$ and a short exact sequence of one of the following forms:
		\begin{equation} \label{eq:essNS}
			\begin{aligned}
				&\dses{\airrns{\mu}^+}{\arelns{\lambda}{\sigma}}{\parrev \airrns{\mu+\fwt}^-}, & &\text{if}\ \mu = +\brac*{\sigma - \tfrac{1}{2}} \fwt \in \lambda, \\
				&\dses{\parrev \airrns{\mu-\fwt}^+}{\arelns{\lambda}{\sigma}}{\airrns{\mu}^-}, & &\text{if}\ \mu = -\brac*{\sigma - \tfrac{1}{2}} \fwt \in \lambda.
			\end{aligned}
		\end{equation}
		\item Similarly, let $\sqrt{1+2q} \notin \ZZ$.  Then, for each $\lambda \in \nsimpseto{\sigma}$, there is a unique solution $\mu$ of $\bilin{\mu}{\mu+2\wvec} = q$ in $\lambda$ and a short exact sequence
		\begin{equation}
			\dses{\airrr{\mu}^+}{\arelr{\lambda}{q}}{\airrr{\mu+\sroot}^-}.
		\end{equation}
	\end{itemize}
\end{proposition}

\section{Application to admissible-level $\ospvoa{\kk}$-modules} \label{sec:voaosp}

We conclude by applying our results to determine exact sequences and (super)characters for the $\arelns{\lambda}{\sigma}$ and $\arelr{\lambda}{q}$ that define modules over the simple affine \svoa{} $\ospvoa{\kk}$, when $\kk$ is admissible:
\begin{equation}
	\kk + \dcox = \frac{u}{2v}, \qquad u \in \ZZ_{\ge 2},\ v \in \ZZ_{\ge 1},\ \frac{u-v}{2} \in \ZZ,\ \gcd \set*{\frac{u-v}{2},v} = 1.
\end{equation}
Note that the dual Coxeter number of $\osp$ is $\dcox = \frac{3}{2}$.  As with the $\asltwo$ case, the Sugawara construction fixes the conformal weights of the ground states of $\arelns{\lambda}{\sigma}$ and $\arelr{\lambda}{q}$ to be
\begin{equation} \label{eq:confwtosp}
	\Delta = \Delta^{\NS}_{\sigma} = \frac{\sigma^2 - 1/4}{4 (\kk + \dcox)} \qquad \text{and} \qquad
	\Delta = \Delta^{\R}_q = \frac{q - \kk/4}{2 (\kk + \dcox)},
\end{equation}
respectively.

The relaxed \hw{} $\ospvoa{\kk}$-modules are classified in \cite{WooAdm18,CreCos18}.  Omitting the \hw{} simples, the classification may be presented in terms of two parameters $r=1,\dots,u-1$ and $s=1,\dots,v-1$, with the module belonging to the \ns{} sector when $r-s$ is odd and to the Ramond sector when $r-s$ is even.  Indeed, the $\arelns{\lambda}{\sigma}$ and $\arelr{\lambda}{q}$ are $\ospvoa{\kk}$-modules when
\begin{equation}
		r-s \in 2\ZZ + 1, \quad \sigma = \sigma_{r,s} = \frac{1}{2} \brac*{r - \frac{u}{v} s} \qquad \text{and when} \qquad
		r-s \in 2\ZZ, \quad q = q_{r,s} = \frac{1}{8} \brac*{r - \frac{u}{v} s}^2 - \frac{1}{2},
\end{equation}
respectively.  We note the ``Kac table''-type symmetries $\sigma_{u-r,v-s} = -\sigma_{r,s}$ and $q_{u-r,v-s} = q_{r,s}$.  Moreover, as $\frac{u-v}{2}$ and $v$ are coprime, we have
\begin{equation}
	\sigma_{r,s} - \frac{1}{2} = \frac{1}{2} \brac*{r-1 - \frac{u}{v} s} = \frac{r-s-1}{2} - \frac{(u-v)/2}{v} s \notin \ZZ
\end{equation}
in the \ns{} sector ($r-s$ odd) and
\begin{equation}
	\sqrt{1+2q_{r,s}} = \frac{1}{2} \abs*{r - \frac{u}{v} s} = \abs*{\frac{r-s}{2} - \frac{(u-v)/2}{v} s} \notin \ZZ
\end{equation}
in the Ramond sector ($r-s$ even).  There are therefore two distinct non-simple relaxed \hwms{} $\arelns{\lambda}{\sigma_{r,s}}$ or $\arelr{\lambda}{q_{r,s}}$, for each $r$ and $s$ (modulo the Kac symmetries).  In particular, the \ns{} solutions to $\norm{\mu}^2 = \frac{1}{2} \brac*{\sigma_{r,s} - \frac{1}{2}}^2$ and the Ramond solutions to $\bilin{\mu}{\mu+2\wvec} = q_{r,s}$ are $\mu = \pm \frac{1}{2} \brac*{r-1 - \frac{u}{v} s} \fwt$ and $\mu = -\wvec \pm \frac{1}{2} \brac*{r - \frac{u}{v} s} \fwt$, respectively.  We therefore define
\begin{equation} \label{eq:defmu}
	\mu_{r,s} =
	\begin{dcases*}
		\frac{1}{2} \brac*{r-1 - \frac{u}{v} s} \fwt, & if $r-s$ is odd, \\
		\frac{1}{2} \brac*{r-2 - \frac{u}{v} s} \fwt, & if $r-s$ is even.
	\end{dcases*}
\end{equation}
Note that $-\mu_{u-r,v-s} = \mu_{r,s} + \fwt$, if $r-s$ is odd, and $-\mu_{u-r,v-s} = \mu_{r,s} + \sroot$, if $r-s$ is even.

\cref{prop:2cfsosp} now gives the $\osp$ analogues of \cref{thm:esvoa}.
\begin{theorem} \label{thm:esvoaosp}
	\leavevmode
	\begin{subequations}
		\begin{itemize}
			\item Each admissible-level $\ospvoa{\kk}$-module $\arelns{\mu_{r,s}}{\sigma_{r,s}}$, where $r=1,\dots,u-1$ and $s=1,\dots,v-1$ satisfy $r-s \in 2\ZZ+1$, is a non-split extension of a (conjugate) simple \hwm{} $\parrev \airrns{\mu_{r,s}+\fwt}^- = \parrev \airrns{-\mu_{u-r,v-s}}^-$ by the simple \hwm{} $\airrns{\mu_{r,s}}^+$.  In other words, the following sequence is exact:
			\begin{equation} \label{eq:esNS}
				\dses{\airrns{\mu_{r,s}}^+}{\arelns{\mu_{r,s}}{\sigma_{r,s}}}{\parrev \airrns{-\mu_{u-r,v-s}}^-}.
			\end{equation}
			\item Similarly, each admissible-level (twisted) $\ospvoa{\kk}$-module $\arelr{\mu_{r,s}}{q_{r,s}}$, where $r=1,\dots,u-1$ and $s=1,\dots,v-1$ satisfy $r-s \in 2\ZZ$, is a non-split extension of a (conjugate) simple \hwm{} $\airrr{\mu_{r,s}+\sroot}^- = \airrr{-\mu_{u-r,v-s}}^-$ by the simple \hwm{} $\airrns{\mu_{r,s}}^+$.  In other words, the following sequence is exact:
			\begin{equation} \label{eq:esR}
				\dses{\airrr{\mu_{r,s}}^+}{\arelr{\mu_{r,s}}{q_{r,s}}}{\airrr{-\mu_{u-r,v-s}}^-}.
			\end{equation}
		\end{itemize}
	\end{subequations}
\end{theorem}

\begin{remark}
	The \ns{} exact sequence \eqref{eq:esNS} follows directly from the first exact sequence of \eqref{eq:essNS}.  If we had instead used the second exact sequence, we would have instead arrived at
	\begin{equation}
		\dses{\parrev \airrns{\mu_{u-r,v-s}}^+}{\arelns{-\mu_{r,s}}{\sigma_{r,s}}}{\airrns{-\mu_{r,s}}^-}.
	\end{equation}
	However, this is seen to be equivalent to \eqref{eq:esNS} by replacing $r$ by $u-r$, $s$ by $v-s$, applying the parity-reversal functor $\parrev$, and using the isomorphism $\parrev \arelns{\lambda}{\sigma} \cong \arelns{\lambda+\fwt}{-\sigma}$ (see \cref{sec:swosp}).
\end{remark}

We now turn to the characters and supercharacters of the $\arelns{\lambda}{\sigma_{r,s}}$, $\lambda \in \h^* / \rlat^0$ and $r-s$ odd, and $\arelr{\lambda}{q_{r,s}}$, $\lambda \in \h^* / \rlat$ and $r-s$ even.  The computations are very similar to that in \cref{sec:voasl2}, reducing the string functions to the negative limiting string functions of $\airrns{\mu_{r,s}}^+$ and $\airrr{\mu_{r,s}}^+$, respectively.  Normalising characters and supercharacters by $\qq^{-\cc/24}$, where
\begin{equation}
	\cc = \frac{\kk}{\kk + \dcox} = 1 - \frac{3v^2}{uv}
\end{equation}
is the central charge of $\ospvoa{\kk}$, \cref{eq:BGG} still holds \cite{IohFus01a} when we replace the $\slvoa{\kk}$-modules by their \ns{} $\ospvoa{\kk}$ analogues.  \cref{prop:vermasfnosp} thus gives the \ns{} string functions:
\begin{align}
	\fsfn{-\infty}{\airrns{\mu_{r,s}}^+}{\qq} &= \sum_{n \in \ZZ} \brac*{\fsfn{-\infty}{\averns{\mu_{2nu+r,s}}^+}{\qq} - \fsfn{-\infty}{\averns{\mu_{2nu-r,s}}^+}{\qq}} \\
	\notag &= \frac{\fjth{2}{1;\qq}}{2 \eta(\qq)^4} \sum_{n \in \ZZ} \brac*{\qq^{\Delta_{2nu+r,s} - \cc/24 + 1/24} - \qq^{\Delta_{2nu-r,s} - \cc/24 + 1/24}} \\
	\notag &= \frac{\fjth{2}{1;\qq}}{2 \eta(\qq)^4} \sum_{n \in \ZZ} \brac*{\qq^{\Delta^{N=1}_{2nu+r,s} - \cc^{N=1}/24} - \qq^{\Delta^{N=1}_{2nu-r,s} - \cc^{N=1}/24}}.
\end{align}
Here, $\Delta_{r,s} = \Delta^{\NS}_{\sigma_{r,s}}$ and the $N=1$ conformal weights and central charge are given by
\begin{equation}
	\Delta^{N=1}_{r,s} = \frac{(vr-us)^2-(v-u)^2}{8uv} + \frac{1}{32} \brac*{1 - (-1)^{r-s}}, \qquad \cc^{N=1} = \frac{3}{2} - \frac{3(v-u)^2}{uv}.
\end{equation}
The link to the $N=1$ superconformal algebra is made manifest through comparing this limiting string function with the character
\begin{equation} \label{eq:RchN=1}
	\chi^{N=1}_{r,s}(\qq) = \frac{\qq^{-\cc^{N=1}/24}}{\eta(\qq)} \sqrt{\frac{\fjth{2}{1;\qq}}{2 \eta(\qq)}} \sum_{n \in \ZZ} \brac*{\qq^{\Delta^{N=1}_{2nu+r,s}} - \qq^{\Delta^{N=1}_{2nu-r,s}}}
\end{equation}
of the simple \emph{Ramond} \hw{} $N=1$ module of conformal weight $\Delta^{N=1}_{r,s}$ and central charge $\cc^{N=1}$.  (Note that $r-s$ odd specifies the Ramond sector of the $N=1$ superconformal minimal models.)

We thereby obtain the (super)characters of the \ns{} relaxed \hwms{}.
\begin{theorem} \label{thm:NSchosp}
	The characters of the admissible-level \ns{} $\ospvoa{\kk}$-modules $\arelns{\lambda}{\sigma_{r,s}}$, with $\lambda \in \h^* / \rlat^0$, $r=1,\dots,u-1$, $s=1,\dots,v-1$ and $r-s \in 2\ZZ+1$, are given by
	\begin{equation} \label{eq:NSchar}
		\fch{\arelns{\lambda}{\sigma_{r,s}}}{\zz;\qq} = \frac{\chi^{N=1}_{r,s}(\qq)}{\eta(\qq)^2} \sqrt{\frac{\fjth{2}{1;\qq}}{2 \eta(\qq)}} \sqbrac*{\sum_{\mu \in \lambda} \zz^{\mu} + \sum_{\mu \in \lambda + \fwt} \zz^{\mu}}.
	\end{equation}
	The supercharacters are given by replacing the sum of the two sums by their difference.
\end{theorem}
\begin{remark}
	We mention that \eqref{eq:RchN=1} is technically not the correct character of the simple \emph{$\ZZ_2$-graded} Ramond $N=1$ module described above because its leading coefficient is $1$, whereas almost all Ramond modules have a two-dimensional space of ground states.  More precisely, $\chi^{N=1}_{r,s}$ is the character of the given simple $N=1$ module when $u,v \in 2\ZZ$, $r=\frac{u}{2}$ and $s=\frac{v}{2}$.  Otherwise, the correct character is obtained by multiplying by $2$.

	Another way of looking at this is to note that while \eqref{eq:RchN=1} is indeed the character of a simple Ramond $N=1$ module, this module only admits a consistent $\ZZ_2$-grading by parity if $u,v \in 2\ZZ$, $r=\frac{u}{2}$ and $s=\frac{v}{2}$.  As far as \cft{} is concerned, these non-$\ZZ_2$-gradable modules are not acceptable in a consistent space of states because they cannot be assigned supercharacters.
\end{remark}

The Ramond (super)characters are a little more subtle to deduce.  Happily, the Ramond version of \cref{eq:BGG} continues to hold.  This does not seem to be mentioned in \cite{IohFus01a}, but is a simple consequence of the existence \cite[Eq.~(3.24)]{RidAdm17} of an invertible functor mapping $\averr{\mu}^+$ to $\averns{\kk \fwt - \mu}^+$.  The subtlety of the computation arises because one has to take into account the relative parity of the Verma submodules of $\averr{\mu_{r,s}}^+$ when determining whether their limiting even or odd string functions contribute to the limiting even or odd string function of $\averr{\mu_{r,s}}^+$ or vice versa.  Indeed, we have
\begin{equation}
	\mu_{2nu+r,s} - \mu_{r,s} = nu \fwt \qquad \text{and} \qquad \mu_{2nu-r,s} - \mu_{r,s} = (nu-r) \fwt,
\end{equation}
hence, by \cref{prop:vermasfnosp}, the limiting string functions must satisfy
\begin{align}
	\frsfn{-\infty}{\pm}{\airrr{\mu_{r,s}}^+}{\qq}
	&= \sum_{n \in \ZZ} \brac*{\frsfn{-\infty}{\pm (-1)^{nu}}{\aver{\mu_{2nu+r,s}}^+}{\qq} - \frsfn{-\infty}{\pm (-1)^{nu-r}}{\aver{\mu_{2nu-r,s}}^+}{\qq}} \\
	\notag &= \frac{\qq^{-\cc/24 + 1/6}}{2 \eta(\qq)^4} \sum_{n \in \ZZ} \brac[\Big]{\qq^{\Delta_{2nu+r,s}} \brac[\big]{\fjth{3}{1;\qq} \pm (-1)^{nu} \fjth{4}{1;\qq}} - \qq^{\Delta_{2nu-r,s}} \brac[\big]{\fjth{3}{1;\qq} \pm (-1)^{nu-r} \fjth{4}{1;\qq}}} \\
	\notag &= \qq^{(3/2 - \cc^{N=1})/24} \biggl[ \frac{\fjth{3}{1;\qq}}{2 \eta(\qq)^4} \sum_{n \in \ZZ} \brac*{\qq^{\Delta^{N=1}_{2nu+r,s}} - \qq^{\Delta^{N=1}_{2nu-r,s}}} \biggr. \\
	\notag &\hphantom{= \qq^{(3/2 - \cc^{N=1})/24}} \biggl. \pm \frac{\fjth{4}{1;\qq}}{2 \eta(\qq)^4} \sum_{n \in \ZZ} (-1)^{nu} \brac*{\qq^{\Delta^{N=1}_{2nu+r,s}} - (-1)^r \qq^{\Delta^{N=1}_{2nu-r,s}}} \biggr],
\end{align}
where $\Delta_{r,s} = \Delta_{q_{r,s}}^{\R}$.  Noting that character and supercharacter of the simple \ns{} \hw{} $N=1$ module of conformal weight $\Delta^{N=1}_{r,s}$ and central charge $\cc^{N=1}$ are
\begin{subequations}
	\begin{align}
		\chi^{N=1}_{r,s}(\qq) &= \frac{\qq^{(3/2-\cc^{N=1})/24}}{\eta(\qq)} \sqrt{\frac{\fjth{3}{1;\qq}}{\eta(\qq)}} \sum_{n \in \ZZ} \brac*{\qq^{\Delta^{N=1}_{2nu+r,s}} - \qq^{\Delta^{N=1}_{2nu-r,s}}} \label{eq:NSchN=1} \\
		\text{and} \qquad \overline{\chi}^{N=1}_{r,s}(\qq) &= \frac{\qq^{(3/2-\cc^{N=1})/24}}{\eta(\qq)} \sqrt{\frac{\fjth{4}{1;\qq}}{\eta(\qq)}} \sum_{n \in \ZZ} (-1)^{nu} \brac*{\qq^{\Delta^{N=1}_{2nu+r,s}} - (-1)^r \qq^{\Delta^{N=1}_{2nu-r,s}}}, \label{eq:NSschN=1}
	\end{align}
\end{subequations}
respectively, the result for the Ramond relaxed \hwms{} follows.
\begin{theorem} \label{thm:Rchosp}
	The admissible-level Ramond $\ospvoa{\kk}$-modules $\arelr{\lambda}{q_{r,s}}$, with $\lambda \in \h^* / \rlat$, $r=1,\dots,u-1$, $s=1,\dots,v-1$ and $r-s \in 2\ZZ$, have the following characters:
	\begin{align} \label{eq:Rchar}
		\fch{\arelr{\lambda}{q_{r,s}}}{\zz;\qq} &= \brac*{\frac{\chi^{N=1}_{r,s}(\qq)}{2 \eta(\qq)^2} \sqrt{\frac{\fjth{3}{1;\qq}}{\eta(\qq)}} + \frac{\overline{\chi}^{N=1}_{r,s}(\qq)}{2 \eta(\qq)^2} \sqrt{\frac{\fjth{4}{1;\qq}}{\eta(\qq)}}} \sum_{\mu \in \lambda} \zz^{\mu} \\
		\notag &\, + \brac*{\frac{\chi^{N=1}_{r,s}(\qq)}{2 \eta(\qq)^2} \sqrt{\frac{\fjth{3}{1;\qq}}{\eta(\qq)}} - \frac{\overline{\chi}^{N=1}_{r,s}(\qq)}{2 \eta(\qq)^2} \sqrt{\frac{\fjth{4}{1;\qq}}{\eta(\qq)}}} \sum_{\mu \in \lambda + \fwt} \zz^{\mu}.
	\end{align}
	The supercharacters are given by replacing the sum of the two sums by their difference.
\end{theorem}

The character formulae of \cref{thm:NSchosp,thm:Rchosp} reduce to the formulae conjectured in \cite[Props.~13 and 14]{RidAdm17} when $\kk = -\frac{5}{4}$, hence $u=2$ and $v=4$.  In this case, the $N=1$ minimal model is trivial, hence the $N=1$ characters and supercharacters appearing in these \lcnamecrefs{thm:NSchosp} are all $1$.

\begin{remark}
	The relaxed character formulae \eqref{eq:NSchar} and \eqref{eq:Rchar} may be somewhat simplified by expressing the elements of the cosets $\lambda$ and $\lambda + \fwt$ explicitly as $\lambda + 2n \fwt$ and $\lambda + (2n+1) \fwt$, respectively, where $n \in \ZZ$:
	\begin{subequations}
		\begin{align}
			\fch{\arelns{\lambda}{\sigma_{r,s}}}{\zz;\qq} &= \zz^{\lambda} \frac{\chi^{N=1}_{r,s}(\qq)}{\eta(\qq)^2} \sqrt{\frac{\fjth{2}{1;\qq}}{2 \eta(\qq)}} \sum_{n \in \ZZ} (\zz^{\fwt})^n, \\
			\fch{\arelr{\lambda}{q_{r,s}}}{\zz;\qq} &= \zz^{\lambda} \sqbrac*{\frac{\chi^{N=1}_{r,s}(\qq)}{2 \eta(\qq)^2} \sqrt{\frac{\fjth{3}{1;\qq}}{\eta(\qq)}} \sum_{n \in \ZZ} (\zz^{\fwt})^n + \frac{\overline{\chi}^{N=1}_{r,s}(\qq)}{2 \eta(\qq)^2} \sqrt{\frac{\fjth{4}{1;\qq}}{\eta(\qq)}} \sum_{n \in \ZZ} (-\zz^{\fwt})^n}.
		\end{align}
	\end{subequations}
	The corresponding supercharacters are now obtained by replacing each $\zz^{\fwt}$ by $-\zz^{\fwt}$ throughout.
\end{remark}

\appendix

\section{Limiting string functions for Verma modules} \label{app:vermasfn}

In this appendix, we detail the computation of the limiting string function in the case of Verma modules.  The results should be expanded in the region $\abs{\qq} < 1$ in order to recover (generalised) formal power series in $\qq$.
\begin{proposition} \label{prop:vermasfnsl2}
	The limiting string function of the Verma $\asltwo$-module $\aver{\mu}^+$ exists and is
	\begin{equation} \label{eq:verlimsfn}
		\fsfn{-\infty}{\aver{\mu}^+}{\qq} = \frac{\qq^{\Delta + 1/8}}{\eta(\qq)^3},
	\end{equation}
	where $\Delta$ is the conformal weight of the ground states of $\aver{\mu}^+$ and $\eta(\qq)$ is Dedekind's eta function.
\end{proposition}
\begin{proof}
	Recall that the character of a Verma $\asltwo$-module is given by
	\begin{equation} \label{eq:chvermasl2}
		\fch{\aver{\mu}^+}{\zz;\qq} = \frac{\zz^{\mu} \qq^{\Delta}}{\prod_{i=1}^{\infty} (1-\zz^{\sroot}\qq^i) (1-\qq^i) (1-\zz^{-\sroot}\qq^{i-1})}.
	\end{equation}
	As string functions are residues (with respect to $\zz^{\sroot}$) of characters, we may write
	\begin{equation}
		\fsfn{\nu}{\aver{\mu}^+}{\qq} = \res_{\zz^{\sroot}} \fch{\aver{\mu}^+}{\zz;\qq} \, \zz^{-\nu-\sroot}
		= \res_{\zz^{\sroot}} \frac{\zz^{\mu-\nu-\sroot} \qq^{\Delta}}{\prod_{i=1}^{\infty} (1-\zz^{\sroot}\qq^i) (1-\qq^i) (1-\zz^{-\sroot}\qq^{i-1})},
	\end{equation}
	where we may convert the \rhs{} into a generalised formal power series in $\zz$ by expanding in the region $1 < \abs{\zz^{\sroot}} < \abs{\qq}^{-1}$.  We extract the factor $(1-\zz^{-\sroot})$ from the denominator of the above expression and note that what remains has an expansion of the form
	\begin{equation} \label{eq:univcharsl2}
		\frac{1}{\prod_{i=1}^{\infty} (1-\zz^{\sroot}\qq^i) (1-\qq^i) (1-\zz^{-\sroot}\qq^i)} = \sum_{n=0}^{\infty} p_n(\zz^{\sroot}) \qq^n,
	\end{equation}
	where each $p_n$ is a Laurent polynomial whose maximal and minimal degrees are $n$ and $-n$, respectively.  (The reader will no doubt recognise \eqref{eq:univcharsl2} as the character of the level-$\kk$ universal \voa{} of $\asltwo$.)

	Since the expansion region requires that $1 < \abs{\zz^{\sroot}}$, we may replace $(1-\zz^{-\sroot})^{-1}$ by a geometric series, thereby arriving at
	\begin{equation}
		\fsfn{\nu}{\aver{\mu}^+}{\qq} = \qq^{\Delta} \sum_{n=0}^{\infty} \sqbrac*{\res_{\zz^{\sroot}} \sum_{m=0}^{\infty} p_n(\zz^{\sroot}) \zz^{\mu-\nu-(m+1)\sroot}} \qq^n.
	\end{equation}
	Here, we have expressed the string function as a (generalised) power series in $\qq$.  As the minimal power of $\zz^{\sroot}$ in $p_n$ is $-n$, the residue gives no contribution unless $m\sroot \ge \mu-\nu-n\sroot$.  It follows that for every fixed order $n$ in the power series, we may choose $\nu$ sufficiently negative so that all contributions to the residue come from $m \ge 0$.  The limit of the string function as $\nu \to -\infty$, $\nu - \mu \in \rlat$, will therefore not be affected if we allow the sum over $m$ to range over all integers.  Recognising $\sum_{m \in \ZZ} \zz^{-m\sroot} = \delta(\zz^{\sroot})$ as a formal delta function and noting that it allows us replace any instance of $\zz^{\beta}$, with $\beta \in \rlat$, by $1$, we obtain the required expression for the limiting string function:
	\begin{align}
		\fsfn{-\infty}{\aver{\mu}^+}{\qq} &= \lim_{\nu \to -\infty} \qq^{\Delta} \sum_{n=0}^{\infty} \sqbrac*{\res_{\zz^{\sroot}} \sum_{m \in \ZZ} p_n(\zz^{\sroot}) \zz^{\mu-\nu-(m+1)\sroot}} \qq^n \\
		&= \lim_{\nu \to -\infty} \res_{\zz^{\sroot}} \frac{\zz^{\mu-\nu-\sroot} \qq^{\Delta}}{\prod_{i=1}^{\infty} (1-\zz^{\sroot}\qq^i) (1-\qq^i) (1-\zz^{-\sroot}\qq^i)} \delta(\zz^{\sroot}) \notag \\
		&= \frac{\qq^{\Delta}}{\prod_{i=1}^{\infty} (1-\qq^i)^3}. \notag \qedhere
	\end{align}
\end{proof}

\begin{proposition} \label{prop:vermasfnosp}
	The limiting string function of the \ns{} Verma $\aosp$-module $\averns{\mu}^+$ exists and is
	\begin{equation} \label{eq:nsverlimsfn}
		\fsfn{-\infty}{\averns{\mu}^+}{\qq} = \qq^{\Delta + 1/24} \frac{\fjth{2}{1;\qq}}{2 \eta(\qq)^4},
	\end{equation}
	where $\Delta$ is the conformal weight of the ground states of $\averns{\mu}^+$ and $\jth{j}$ denotes the Jacobi theta functions.

	For the Ramond Verma $\aosp$-module $\averr{\mu}^+$, the limiting even and odd string functions exist and are
	\begin{equation} \label{eq:rverlimsfn}
		\frsfn{-\infty}{\pm}{\averr{\mu}^+}{\qq} = \frac{\qq^{\Delta + 1/6}}{2 \eta(\qq)^4} \brac[\big]{\fjth{3}{1;\qq} \pm \fjth{4}{1;\qq}},
	\end{equation}
	where $\Delta$ now denotes the conformal weight of the ground states of $\averr{\mu}^+$.
\end{proposition}
\begin{proof}
	The character of a \ns{} Verma $\aosp$-module is
	\begin{equation}
		\fch{\averns{\mu}^+}{\zz;\qq} = \zz^{\mu} \qq^{\Delta} \prod_{i=1}^{\infty} \frac{(1+\zz^{\fwt}\qq^i) (1+\zz^{-\fwt}\qq^{i-1})}{(1-\zz^{\sroot}\qq^i) (1-\qq^i) (1-\zz^{-\sroot}\qq^{i-1})}.
	\end{equation}
	The derivation of \eqref{eq:nsverlimsfn} now mirrors that of \eqref{eq:verlimsfn} except that we extract the factor
	\begin{equation}
		\frac{1+\zz^{-\fwt}}{1-\zz^{-\sroot}} = \frac{1}{1-\zz^{-\fwt}} = \sum_{m=0}^{\infty} \zz^{-m\fwt}.
	\end{equation}
	Again, we check that it is permissible to replace this geometric sum by the formal delta function $\delta(\zz^{\fwt})$ when considering the limiting string function.  The result now follows using standard identities for theta functions (for which we use the conventions of \cite[App.~B]{RidSL208}).

	The character of a Ramond Verma $\raosp$-module is instead
	\begin{equation}
		\fch{\averr{\mu}^+}{\zz;\qq} = \zz^{\mu} \qq^{\Delta} \prod_{i=1}^{\infty} \frac{(1+\zz^{\fwt}\qq^{i-1/2}) (1+\zz^{-\fwt}\qq^{i-1/2})}{(1-\zz^{\sroot}\qq^i) (1-\qq^i) (1-\zz^{-\sroot}\qq^{i-1})}.
	\end{equation}
	This time, we can only extract
	\begin{equation} \label{eq:Rextract}
		\frac{1}{1-\zz^{-\sroot}} = \sum_{m=0}^{\infty} \zz^{-2m\fwt} = \frac{1}{2} \sum_{m=0}^{\infty} \sqbrac[\big]{(\zz^{-\fwt})^m + (-\zz^{-\fwt})^m}
	\end{equation}
	which gets replaced by $\frac{1}{2} \brac*{\delta(\zz^{\fwt}) + \delta(-\zz^{\fwt})}$.  The limiting even string function now follows from the usual manipulations.  To get the limiting odd string functions, we multiply \eqref{eq:Rextract} by $\zz^{-\fwt}$ so that the replacement is instead by $\frac{1}{2} \brac*{\delta(\zz^{\fwt}) - \delta(-\zz^{\fwt})}$.
\end{proof}

\flushleft
\bibliography{relax}
\bibliographystyle{unsrt}

\end{document}